\documentclass{amsart}
\usepackage{graphicx}
\usepackage{amssymb}
\usepackage{array}
\usepackage{enumerate}
\usepackage{url}
\usepackage{algorithm}
\usepackage{verbatim}

\DeclareMathOperator{\sqf}{sqf} 
\DeclareMathOperator{\cbf}{cbf} 
\DeclareMathOperator{\spf}{spf} 
 
\DeclareMathOperator{\sgn}{sgn}

\newtheorem{theorem}{Theorem}

\newtheorem{lemma}[theorem]{Lemma}

\newtheorem{corollary}[theorem]{Corollary}
\numberwithin{theorem}{section}
\theoremstyle{remark}
\newtheorem*{remark}{Remark}

\makeatletter
\let\@@pmod\pmod
\DeclareRobustCommand{\pmod}{\@ifstar\@pmods\@@pmod}
\def\@pmods#1{\mkern 2mu({\operator@font mod}\mkern 3mu#1)}
\makeatother

\begin{document}

\title[An explicit hybrid estimate for $L(1/2+it,\chi)$]{An explicit hybrid estimate
for $L(1/2+it,\chi)$}
\author[G.A. Hiary]{Ghaith A. Hiary}
\thanks{Preparation of this material is partially supported by
the National Science Foundation under agreements No. 
 DMS-1406190.}
\address{Department of Mathematics, The Ohio State University, 231 West 18th
Ave, Columbus, OH 43210.}
\email{hiary.1@osu.edu}
\subjclass[2010]{Primary 11Y05.}
\keywords{Van der Corput method, Weyl method, exponential sums, Dirichlet
$L$-functions, powerfull modulus, hybrid estimates, explicit estimates.}

\begin{abstract}
An explicit hybrid estimate for $L(1/2+it,\chi)$ is derived, where 
$\chi$ is a Dirichlet character modulo $q$. 
The estimate applies when $t$ is bounded away from zero, and 
is most effective when $q$ is powerfull, 
yielding an explicit Weyl bound in this case.
The estimate takes a particularly simple form if $q$ is a sixth power.
Several hybrid lemmas of 
van der Corput--Weyl type are presented.
\end{abstract}

\maketitle

\section{Introduction} \label{intro}

Let $\chi$ be a Dirichlet character modulo $q$. 
Let $L(1/2+it,\chi)$ be the corresponding Dirichlet $L$-function
on the critical line.
Let $\tau(q)$ be the number of divisors of $q$. 
If $|t|\ge 3$, say, we define 
the analytic conductor of $L(1/2+it,\chi)$
to be $\mathfrak{q}:=q|t|$.
 
 We are interested in finding
an explicit hybrid estimate 
for $L(1/2+it,\chi)$ in terms of $\mathfrak{q}$ and $\tau(q)$. 
Specifically, we would like to find constants 
$c$, $\kappa_1$, $\kappa_2$, $\kappa_3$, and $t_0\ge 3$ as small 
as possible, such that 
\begin{equation}\label{sought bound}
|L(1/2+it,\chi)| \le
c\, \tau(q)^{\kappa_1}\mathfrak{q}^{\kappa_2} \log^{\kappa_3}\mathfrak{q}, \qquad
(|t|\ge t_0).
\end{equation}

If $|t|\le t_0$, then estimating $L(1/2+it,\chi)$
reduces, essentially, to bounding
pure character sums.
Barban, Linnik, and Tshudakov~\cite{blt} gave Big-$O$ bounds 
for such sums, as well as some applications.

The convexity bound in our context is $L(1/2+it,\chi)\ll 
\mathfrak{q}^{1/4}$.
This can be derived using the standard method of the 
approximate functional equation. 
Habsieger derived such an approximate equation 
in \cite{habsieger}.
And we use this
in \textsection{\ref{bounds proofs}} to prove that
if $\chi$ is a primitive character\footnote{We consider the principal character
as neither primitive nor imprimitive.}
modulo $q>1$, then we have the convexity bound
\begin{equation}\label{convexity bound}
|L(1/2+it,\chi)| \le 124.46 \mathfrak{q}^{1/4},\qquad
(\mathfrak{q} \ge 10^9,\,\, |t|\ge \sqrt{q}).
\end{equation}
Previously, Rademacher~\cite{rademacher} derived the explicit bound
\begin{equation}
|L(\sigma+it,\chi)| \le
\left(\frac{q|1+\sigma+it|}{2\pi}\right)^{\frac{1+\eta-\sigma}{2}}\zeta(1+\eta),
\end{equation}
valid if $0<\eta\le 1/2$, $\sigma\le 1+\eta$, and $\chi\pmod q$ is
primitive. This is nearly a convexity bound
except for an additional $\eta>0$ in the exponent.

Using partial summation, 
we obtain an explicit bound applicable for any $t$. 
Specifically, if $\chi$ is primitive modulo $q>1$, 
then we obtain in  \textsection{\ref{bounds proofs}} that 
\begin{equation}\label{partial summation bound}
|L(1/2+it,\chi)|\le 4 q^{1/4}\sqrt{(|t|+1)\log q}.
\end{equation}
The bound \eqref{partial summation bound} 
is weaker than the convexity bound in general, 
but it can be useful in the limited region where $t$ is small.

Our main result is Theorem~\ref{main theorem}.
This theorem supplies the first example of 
an explicit hybrid Weyl bound 
(i.e.\ with $\kappa_2=1/6$ in \eqref{sought bound})
for an infinite set of Dirichlet $L$-functions; namely, 
the set of Dirichlet $L$-functions corresponding to powerfull moduli.
Theorem~\ref{main theorem} takes a particularly
simple form if  
$q$ is a sixth power and $\chi$ is primitive, 
yielding Corollary~\ref{main theorem simple} below. 
\begin{corollary}\label{main theorem simple}
Let $\chi$ be a primitive Dirichlet character modulo $q$. 
If $q$ is a sixth power, then
\begin{equation}
|L(1/2+it,\chi)| \le 9.05 \tau(q) \mathfrak{q}^{1/6}
\log^{3/2} \mathfrak{q},\qquad (|t|\ge 200).
\end{equation}
\end{corollary}

In the notation of \eqref{sought bound}, 
Corollary~\ref{main theorem simple} asserts that
if $q$ is a sixth power and $\chi$ is primitive, then the choice 
$c=9.05$, $\kappa_1=1$, $\kappa_2=1/6$, $\kappa_3=3/2$, and $t_0=200$
is admissible. The constant $\kappa_3=3/2$ arises
from two sources: a dyadic division 
that contributes $1$, and the Weyl 
differencing method (see \cite[\textsection{5.4}]{titchmarsh}) 
which contributes $1/2$. 
The constant $\kappa_1=1$ arises, in part, when counting
the number of solutions to quadratic congruence equations 
in the Weyl differencing method.
The $\kappa_2=1/6$ arises from proving that, on average, square-root
cancellation occurs in certain short 
segments of the dyadic pieces 
$\sum_{V\le n< 2V}\frac{\chi(n)}{n^{1/2+it}}$.
The constant $c=9.05$ is largely contributed by the part of the main 
sum over $\mathfrak{q}^{1/3}\ll n\ll \mathfrak{q}^{2/3}$.
Last, the constant $t_0=200$ is due to technical reasons, 
and can be lowered with some work.

We state the main theorem below.
See \textsection{\ref{main notation}} for the definitions of 
$\sqf(q)$, $\cbf(q)$, $\spf(q)$, $B$, $B_1$, $D$, and $\Lambda(D)$. 
For now we remark that if $\chi$ is primitive, then $B=B_1=1$. 
And if $q$ is sixth power, then
$\sqf(q)=\cbf(q)=\spf(q) = 1$. The number $\Lambda(D)$ is bounded by $\tau(D)$,
and $D$ is usually of size about $q^{1/3}$.

\begin{theorem}\label{main theorem}
Let $\chi$ be a Dirichlet character modulo $q$.
If $|t|\ge 200$, then
\begin{equation}
|L(1/2+it,\chi)| \le \mathfrak{q}^{1/6} Z(\log \mathfrak{q})+W(\log
\mathfrak{q})
\end{equation}
where
\begin{equation*}
\begin{split}
Z(X)&:= 6.6668 \sqrt{\cbf(q)} - 16.0834 \spf(q) +15.6004 \spf(q)X\\
&+1.7364\sqrt{\Lambda(D) \cbf(q)(65.5619 - 17.1704 X - 2.4781 X^2 + 0.6807 X^3)}\\
&+1.7364 \sqrt{\Lambda(D) \cbf(q) B \tau(D/B) (-1732.5 - 817.82 X +71.68 X^2
+ 47.57 X^3)},\\
&\\
W(X)&:= -101.152 - 195.696 B_1 \sqf(q)+ 19.092 X + 94.978 B_1 \sqf(q) X.
\end{split}
\end{equation*}
\end{theorem}

For many applications, 
it suffices to focus on the case where 
$\chi$ is primitive.
For if not, then letting $\chi_1\pmod{q_1}$ be the primitive 
character inducing $\chi$ and using 
the Euler product  and analytic continuation of $L(s,\chi)$,
we have
\begin{equation}\label{imprimitive to primitive}
|L(1/2+it,\chi)| 
\le |L(1/2+it,\chi_1)|\prod_{\substack{p|q\\ p\nmid q_1}}(1+1/\sqrt{p}).
\end{equation}
Thus, we obtain an explicit bound on $L(1/2+it,\chi)$
by bounding $L(1/2+it,\chi_1)$
and using the inequality \eqref{imprimitive to primitive}. 
In our proof of Theorem~\ref{main theorem}, though, 
we bound general sums of the form \eqref{general f sum},
and keep track of the dependence on $B$ and $B_1$.

The main devices in our proofs 
are the hybrid van der Corput--Weyl
Lemmas \ref{corput lemma 1 1} and \ref{corput lemma 2 2}. 
These lemmas provide explicit bounds for sums of the form 
\begin{equation}\label{general f sum}
\sum_{n=N+1}^{N+L} \chi(n)e^{2\pi i f(n)},
\end{equation}
where we take $f(x) = -\frac{t}{2\pi}\log x$  in our application. 
A pleasant feature of the resulting bounds 
is that they naturally split into 
two main terms, one originating from $\chi(n)$ 
and the other originating from $n^{-iqt}$. 
In particular, we can detect cancellation in the $q$ and
$t$ aspects separately, then 
combine the savings routinely
using the well-spacing Lemma~\ref{well spacing lemma}.

The starting point in our proof of Theorem~\ref{main theorem} 
is the Dirichlet series
\begin{equation}\label{dirichlet series}
L(1/2+it,\chi) = \sum_{n=1}^{\infty} \frac{\chi(n)}{n^{1/2+it}},
\end{equation}
valid for $\chi$ nonprincipal. (If $\chi$ is principal, we use a bound 
for the Riemann zeta function.) We partition the sum in \eqref{dirichlet series} 
into four parts: $1\ll n\ll
\mathfrak{q}^{1/3}$ which is bounded trivially, 
$\mathfrak{q}^{1/3}\ll n\ll \mathfrak{q}^{2/3}$ for which
Lemma~\ref{corput lemma 1 1} is used,
$\mathfrak{q}^{2/3} \ll n\ll \mathfrak{q}$
for which Lemma~\ref{corput lemma 2 2} is used, and
 the tail $\mathfrak{q}\ll n$ which is bounded using the P\'olya--Vinogradov 
inequality. 

We remark that the restriction in Corollary~\ref{main theorem simple} 
that $q$ is a sixth
power may be relaxed to $q$ is a cube provided that one starts with a main sum of
length about $\sqrt{\mathfrak{q}}$ (as in the approximate functional equation)
instead of the main sum \eqref{dirichlet series}. 
One then applies van der Corput lemmas
analogous to those in \cite{hiary-corput}, but for the twisted sums
\eqref{general f sum}. 

Interest in powerfull modulus $L$-functions has
grown recently, both from theoretical and computational
perspectives. 
Mili\'cevi\'c~\cite{milicevic} 
has recently derived sub-Weyl bounds for pure character sums to prime-power modulus.
And the author~\cite{hiary-char-sums} had derived an algorithm
to compute hybrid sums to
powerfull modulus in $\mathfrak{q}^{1/3+o(1)}$ time. 
If $q$ is smooth (but not necessarily powerfull) 
or prime, then one 
can obtain explicit hybrid subconvexity bounds
by deriving 
an explicit version of Heath-Brown's $\mathfrak{q}$-analogue of 
the van der Corput method
in \cite{heath-brown-1}, and an explicit version of 
Heath-Brown's hybrid Burgess
method in \cite{heath-brown-2}. 

\section{Notation}\label{main notation}
Let $\chi$ be a Dirichlet character modulo $q$.
We factorize the modulus 
\begin{equation}
q := p_1^{a_1}\cdots p_{\omega}^{a_{\omega}},
\end{equation}
where the $p_j$ are distinct primes and $a_j\ge 1$.
For each prime power $p^a$, we define  
\begin{equation}
C_1(p^a):=p^{\lceil a/2\rceil},\qquad D_1(p^a):=p^{a-\lceil a/2\rceil},
\end{equation}
then extend the definitions multiplicatively; i.e.\  
$C_1(q) = C_1(p_1^{a_1})\cdots C_1(p_{\omega}^{a_{\omega}})$. 
In addition, we define
\begin{equation}
C(p^a):= p^{\lceil a/3\rceil},
\qquad
D(p^a):=\left\{\begin{array}{ll}
1 & a = 1,\\
p^{a-2\lceil a/3\rceil+1} & p = 2 \textrm{ and } a > 1,\\
p^{a-2\lceil a/3\rceil} & p\ne 2 \textrm{ and } a > 1,
\end{array}\right.
\end{equation}
then extend the definitions multiplicatively.
Since the quantities $C_1(q)$, $D_1(q)$, $C(q)$, and $D(q)$ will appear often,
it is useful to introduce the short-hand notation
$C_1:= C_1(q)$, $D_1:=D_1(q)$, $C:=C(q)$, and $D:=D(q)$. 
For example, $C_1D_1=q$. 

Some additional arithmetic factors will appear in our estimates: 
 $(m,n)\ge 0$ is the greatest common divisor 
of $m$ and $n$,
$\omega(m)$ is the number of distinct prime divisors of $m$, and 
$\Lambda(m)$ is the number of solutions of the congruence
$x^2\equiv 1\pmod m$ with $0\le x < m$. Explicitly, 
\begin{equation}
\Lambda(m)= \left\{\begin{array}{ll}
2^{\omega(m) -1}, & m \equiv 2\pmod 4,\\
2^{\omega(m)},& m\not \equiv 2\pmod{4} \textrm{ and } m\not\equiv 0\pmod 8,\\
2^{\omega(m)+1}, & m\equiv 0\pmod 8.
\end{array}\right.
\end{equation}
We define $\Lambda := \Lambda(D)$, and 
\begin{equation}\label{prf def}
\begin{split}
&\sqf(p^a) := p^{\lceil a/2\rceil -a/2},\qquad
\cbf(p^a) :=p^{\lceil a/3\rceil - a/3}, \\
&\spf(p^a) := \frac{p^{\lceil a/2\rceil -\lceil a/3\rceil/2 - a/6}}
{\sqrt{D(p^a)}},
\end{split}
\end{equation}
then extend the definitions multiplicatively.
Note that $\sqf(q)$ is determined by the primes $p_j|q$ such that 
$a_j\not\equiv 0\pmod 2$ and 
$\cbf(q)$ by the primes $p_j$ such that 
$a_j\not\equiv 0\pmod 3$.
If $q$ is a square, then $\sqf(q)=1$. If $q$ is a cube, then $\cbf(q)=1$.
And if $q$ is a sixth power, then $\sqf(q)=\cbf(q)=1$ 
and $\spf(q)\le 1$.

The numbers $B$ and $B_1$ that appear
in Theorem~\ref{main theorem} are defined in 
Lemma~\ref{postnikov lemma}.

In the remainder of the paper, we use the following notation: 
$\exp(x)=e^x$ is the usual exponential function, 
$[x]$ is the closest integer to $x$, 
$\|x\|$ is the distance to the closest integer to $x$, $\bar{\ell}\pmod C$
is the modular inverse of $\ell \pmod C$ if it exists, and
\begin{equation}
\sgn(x)=\left\{\begin{array}{ll}
1, &x\ge 0,\\
0, & x< 0.
\end{array}\right.
\end{equation}
\\

\noindent
\textbf{Acknowledgments.} I thank Tim Trudgian for pointing out 
the reference \cite{rademacher}.

\section{Preliminary Lemmas}
\begin{lemma}\label{well spacing lemma}
Let $\{y_r : r=0,1,\ldots\}$ be a set of real numbers. 
Suppose that there is a number $\delta >0$ such that 
$\min_{r\ne r'} |y_r-y_{r'}|\ge \delta$.
If $P\ge 2$ and $y\ge x$ then
\begin{equation}\label{well spacing lemma 1}
\sum_{y_r\in [x,y]} \min(P,\|y_r\|^{-1})\le
2(y-x+1)(2P+\delta^{-1}\log(eP/2)).
\end{equation}
If $P< 2$, then replace the r.h.s.\ by $2(y-x+1)(P+\delta^{-1})$.
\end{lemma}
\begin{proof}
We may assume that $\delta \le 1/2$, otherwise the bounds follow 
on trivially estimating the number of points $y_r$ in $[x,y]$ by $2(y-x)+1$ 
and  using the trivial bound $\min(P,\|y_r\|^{-1}) \le P$.

For each integer $k\in [x, y]$, we consider the interval 
$[k-1/2,k+1/2]$. There are at most two points
$y_r$ in $[k-\delta,k+\delta)$, say $y_k^+\in [k,k+\delta)$ 
and $y_k^-\in [k-\delta,k)$. 
If no such points exist, then we insert 
one (or both) of them
subject to the condition $|y_k^+-y_k^-| \ge \delta$. 
To preserve the $\delta$-spacing condition, we slide the remaining 
points in $(y_k^+,k+1/2]$ (resp. $[k-1/2,y_k^{-})$) to the right of $y_k^{+}$
(resp. left of $y_k^-$) in the obvious way. 
It is possible that a point falls off each edge, in which
case we may discard it.
This is permissible since the overall procedure that we described can 
only increase the magnitude of the sum in \eqref{well spacing lemma 1}.

We have $y_k^+ = k+\rho_k \delta$ for some
$\rho_k \in [0,1)$, and so
$y_k^- \le k+(\rho_k-1)\delta$.
Hence, using the inequality 
$\min(P,\|y_r\|^{-1})+\min(P,\|y_{r'}\|^{-1})\le \min (2P,\|y_r\|^{-1} +
\|y_{r'}\|^{-1})$, and the formula $\|y_r\| = |y_r-k|$ if $|y_r - k| \le 1/2$, 
we obtain
\begin{equation}\label{well spacing lemma 2}
\begin{split}
\sum_{|y_r-k|\le \frac{1}{2}} \min(P,\|y_r\|^{-1})
\le \sum_{0\le r\le \frac{1}{2\delta}} 
\min\left(2P,\frac{1}{\delta(r+\rho_k)}+\frac{1}{\delta(r+1-\rho_k)}\right).
\end{split}
\end{equation}
We observe that
\begin{equation}\label{well spacing lemma 3}
\frac{1}{\delta(r+\rho_k)}+\frac{1}{\delta(r+1-\rho_k)}
= \frac{1}{\delta} \frac{2r+1}{r^2+r+\rho_k-\rho_k^2}\le \frac{2}{\delta r}.
\end{equation}
Combining this with the observation 
$\frac{2}{r\delta} \ge 2P$ if $r\le \frac{1}{\delta P}$, we conclude
that 
\begin{equation}\label{well spacing lemma 4}
\sum_{|y_r-k|\le \frac{1}{2}} \min(P,\|y_r\|^{-1})
\le 2P\lceil 1/\delta P\rceil 
+ \sum_{\lceil 1/\delta P\rceil \le r\le 1/2\delta} \frac{2}{r\delta}.
\end{equation}
To bound the sum over $r$, we isolate the first term
and estimate the remainder by an integral. 
If $P\ge 2$ (so that the integral below makes sense), 
then this gives the bound 
\begin{equation}
\sum_{\lceil 1/\delta P\rceil \le r\le 1/2\delta} \frac{2}{r\delta}\le 
2P+2\delta^{-1}\int_{1/\delta P}^{1/2\delta} \frac{1}{x}dx.
\end{equation}
The integral evaluates to $\log(P/2)$. Therefore, the r.h.s.\ in
\eqref{well spacing lemma 4} is bounded by 
$2\delta^{-1} + 2P+2P + 2\delta^{-1}\log(P/2)$. 
So the lemma follows if $P\ge 2$ as the cardinality of
$\{k : x\le k\le y\}$ is  $\le y-x+1$.
Finally, if $P< 2$, then the sum on the r.h.s.\ in \eqref{well spacing lemma 4} 
is empty, and so the bound is $2\delta^{-1} + 2P$. 
\end{proof}

\begin{lemma}\label{exp reduce}
Let $f$ be an analytic function on a disk of radius
$\lambda(L-1)$ centered at $N+1$, where $\lambda > 1$ and $1\le L \in \mathbb{Z}$.
If there is a number $\eta$ and an integer $J \ge 0$ such that
$\frac{|f^{(j)}(N+1)|}{j!}\lambda^j(L-1)^j\le 
\frac{\eta}{\lambda^j}$ for $j> J$,
then
\begin{equation}
\left|\sum_{n=N+1}^{N+L} \chi(n) e^{2\pi i f(n)} \right| \le  
\nu_J(\lambda,\eta) \max_{0\le \Delta < L} 
\left|\sum_{n=N+1+\Delta}^{N+L} \chi(n)e^{2\pi i
P_J(n-N-1)}\right|, 
\end{equation}
where 
\begin{equation}
\begin{split}
P_J(x) := \sum_{j=0}^J \frac{f^{(j)}(N+1)x^j}{j!},\quad 
\nu_J(\lambda,\eta):= \left(1+\frac{\lambda^{-J}}{\lambda-1}\right)
\exp \left(\frac{2\pi \eta \lambda^{-J}}{\lambda-1}\right).
\end{split}
\end{equation}
\end{lemma}
\begin{proof}
If $L=1$, the lemma is trivial. So assume that $L>1$. 
We apply the Taylor expansion to obtain
\begin{equation}
f(N+1+z) = P_J(z)+\sum_{j>J}
\frac{f^{(j)}(N+1)}{j!} z^j, \qquad (|z| \le \lambda(L-1)).
\end{equation}
Using the Taylor expansion once more,
\begin{equation}
e^{2\pi i (f(N+1+z)-P_J(z))}=:\sum_{j=0}^{\infty} c_j(J,N) z^j, \qquad 
(|z| \le \lambda(L-1)).
\end{equation}
So if we define $\nu^*_J :=  \sum_{j=0}^{\infty} |c_j(J,N)(L-1)^j|$,
then partial summation gives
\begin{equation}\label{weyl 0 est}
\left|\sum_{n=N+1}^{N+L} \chi(n) e^{2\pi i f(n)} \right| \le
\nu^*_J \max_{0\le \Delta< L} 
\left|\sum_{n=N+1+\Delta}^{N+L} \chi(n) e^{2\pi i P_J(n-N-1)}\right|.
\end{equation}
To estimate the coefficients $c_j(J,N)$, we 
use the Cauchy theorem 
applied with a circle of radius $\lambda(L-1)$ around the origin. 
In view of the growth condition on the derivatives of $f$, 
this yields
\begin{equation}
|c_j(J,N)| \le \frac{1}{2\pi}\left|\oint \frac{e^{2\pi i
(f(N+1+z)-P_J(z))}}{z^{j+1}}\,dz\right| 
\le \frac{\exp\left(\frac{2\pi \eta\lambda^{-J}}{\lambda-1}\right)}{\lambda^j(L-1)^j}.
\end{equation}
Noting that $c_j(J,N)=0$ for $1\le j\le J$, we therefore deduce that 
\begin{equation}
\nu^*_J \le \exp\left(\frac{2\pi \eta\lambda^{-J}}{\lambda-1}\right)
\left[1+ \sum_{j>J} \frac{(L-1)^j}{\lambda^j(L-1)^j}\right]
=\nu_J(\lambda,\eta).
\end{equation}
\end{proof}

\begin{lemma}\label{postnikov lemma}
There exists an integer $\tilde{L}$ such that
$\chi(1+C_1 x) = e^{2\pi i \tilde{L} x/D_1}$ 
for all $x\in \mathbb{Z}$. 
If $\chi$ is primitive, then $B_1:=(\tilde{L},D_1)=1$.
Furthermore, there exist integers $L_0$ and $L$ such that
$\chi(1+C x) = e^{4\pi i L_0 x/CD + 2\pi i L x^2/D}$ 
for all $x\in \mathbb{Z}$.
If $\chi$ is primitive, then
 $L$ can be chosen so that $B:=(L,D)=1$. 
\end{lemma}
\begin{proof}
We start with the decomposition 
$\chi = \chi_1\cdots\chi_{\omega}$, where $\chi_j$ is a Dirichlet
character modulo $p_j^{a_j}$. 
By \cite[Lemma 3.4]{hiary-simple-alg}, there
exists an integer $\tilde{L}_j$ such that 
\begin{equation}
\chi_j(1+C_1(p_j^{a_j})x) = e^{2\pi i
\tilde{L}_j x/D_1(p_j^{a_j})}
\end{equation}
for all $x\in \mathbb{Z}$. Hence,
\begin{equation}\label{C1 form}
\chi(1+C_1 x) = \chi_1(1+C_1x) \cdots \chi_{\omega}(1+C_1x)=e^{2\pi i
\tilde{L} x/D_1},
\end{equation}
where, noting that $C_1D_1 = q$, we have
\begin{equation}
\tilde{L}= q\sum_{j=1}^{\omega} \tilde{L}_j/p_j^{a_j}.
\end{equation}
Let $B_1=(\tilde{L},D_1)$.
It is clear that
$\chi(1+ qx/B_1)=1$ for all $x$. 
So $q/B_1$ is an induced modulus for $\chi$. 
In particular, if $B>1$ then $\chi$ is imprimitive.
This completes the proof of the first part of the lemma.

For the second part, we use
\cite[Lemma 4.2]{hiary-char-sums}. 
Consider first the case $p_j^{a_j}\not\in\{4,8\}$ and $a_j>1$. 
Then there are integers $L_{0,j}$ and  $L_j$ such that 
\begin{equation}\label{assertion 1}
\chi_j(1+C(p_j^{a_j})x) = e^{\frac{4\pi i L_{0,j}  x}{C(p_j^{a_j})D(p_j^{a_j})}+ 
\frac{2\pi i L_j x^2}{D(p_j^{a_j})}}
\end{equation}
for all $x\in \mathbb{Z}$, and moreover we can take 
$L_{0,j}=- L_j$.
If $a_j=1$,
then $C(p_j^{a^j})=p_j^{a_j}$. So $\chi_j(1+C(p_j^{a_j})x) =1$ 
and we can take $L_{0,j}=L_j=0$.
If $p_j^{a_j} = 4$, then either $L_{0,j}=0$ and $L_j=1$ or $\chi$ is principal.
If $p_j^{a_j} = 8$, then either $L_{0,j} = L_j=1$, or $L_{0,j} = 2$ and $L_j=0$
(an imprimitive character), 
or $L_{0,j}=-1$ and $L_j=1$, or $\chi$ is principal.
Put together, we have 
\begin{equation}
\chi(1+C x) = \chi_1(1+Cx) \cdots \chi_{\omega}(1+Cx)=
e^{\frac{4\pi i  L_0 x}{CD} + \frac{2\pi i L x^2}{D}}
\end{equation}
where
\begin{equation}\label{cond L L*}
\begin{split}
L_0 = C^2D\sum_{j=1}^{\omega} \frac{L_{0,j}}{C(p_j^{a_j})^2D(p_j^{a_j})},\qquad
L = C^2D\sum_{j=1}^{\omega} \frac{L_j}{C(p_j^{a_j})^2D(p_j^{a_j})}.
\end{split}
\end{equation}

It remains to prove that if $\chi$ is primitive then $B=1$. 
To this end, we note that 
$\frac{Lq^2}{B^2C^2D}$ is an integer.
So if we show that $\frac{2L_0q}{BC^2D}$ is an integer too,
then $\chi(1+qx/B)=1$ for all $x\in \mathbb{Z}$.
In particular, if $B>1$, then $q/B$ is a nontrivial
induced modulus and $\chi$ is imprimitive, 
which completes the proof of the second
part of the lemma. 

Now, to show that $\frac{2L_0q}{BC^2D}$ is an integer, 
we first note that $\frac{L_0q}{C^2}$ is always an integer. (Recall that
$L_0=0$ if $a_j=1$.) 
Furthermore, if $a_j=1$ then $(D,p_j)=1$ and so $(B,p_j)=1$.
In light of this, we may assume that $a_j>1$ for all $j$. 

We consider two possibilities. 
If $p_j^{a_j}\not\in\{4,8\}$ for any $j$, then $C^2D=q$ (if $q$ is odd)
or $2q$ (if $q$ is even), and in any case 
$L_{0,j} = -L_j$ for all $j$. The last fact implies in turn that
$L_0=-L$, hence $B=(L_0,D)$. In particular, $B$ divides $L_0$ 
and we conclude that $\frac{2L_0q}{BC^2D}=\frac{L_0}{B}$ or $\frac{2L_0}{B}$ 
and so is an integer in either case. 

On the other hand, 
if $p_j^{a_j}\in \{4,8\}$ for some $j$, then $C^2D=2q$ and we appeal to  
the remark following \eqref{assertion 1}. 
Accordingly, if $\chi$ is primitive and 
$p_j^{a_j}\in \{4,8\}$ then 
$L_j=1$ and so $L$ must be odd. 
This shows that $B=(L,D/2)$. In addition, we have 
\begin{equation}
L_0 = L-\left\{\begin{array}{ll}
\frac{C^2(L_j-L_{0,j})}{4}\frac{D}{2} & p_1^{a_1}=4,\\
\frac{C^2(L_j-L_{0,j})}{8}\frac{D}{2} & p_1^{a_1}=8.
\end{array}\right.
\end{equation}
Therefore, given the possibilities for $L_{0,j}$ and $L_j$ stated after
\eqref{assertion 1}, we see if $\chi$ is primitive then
$L_0\equiv L \pmod{D/2}$, and so $B= (L_0,D/2)$. 
This shows that $B$ is a divisor of $L_0$, hence 
$\frac{2L_0q}{BC^2D}=\frac{L_0}{B}$ is an integer.
\end{proof}

\begin{lemma}\label{gcd sum lemma}
Let $M,N\in \mathbb{Z}_{\ge 1}$, $W_M(m) := 1-m/M$,
and $d_m(N):= (2m,N)$. Then
\begin{equation}\label{gcd sum eq 1}
\begin{split}
\sum_{m=1}^M W_M(m) \frac{d_m(N)}{m} \le \tau(N) \log M,\qquad
\sum_{m=1}^M W_M(m) d_m(N) \le \tau(N) M.
\end{split}
\end{equation}
\end{lemma}
\begin{proof}
We prove the first bound, the second one being analogous. 
Let us write $N = 2^aN'$ with $N'$ odd. We induct on $a$. 
If $a=0$, then $d_m(N) = (m,N)$ and the result follows  
because
\begin{equation}\label{gcd sum lemma 1}
\begin{split}
\sum_{m=1}^M W_M(m) \frac{d_m(N)}{m}  &\le \sum_{\substack{r|N\\ r\le 2M}}
\sum_{1\le m'\le M/r} W_M(rm') \frac{1}{m'}\\
&\le \tau(N) \sum_{1\le m'\le M} W_M(m')\frac{1}{m'}\\
&\le \tau(N)\log M.
\end{split}
\end{equation}
If $a=1$, then 
$d_m(N) = 2 d_m(N')$. 
So using the previous calculation 
and observing that $2\tau(N') = \tau(N)$ yields the desired bound.

Henceforth, we may assume that $a\ge 2$.
We may further assume that $M> 2$, for if $M=1$ or $2$ then the lemma is trivial.

Since $N$ is even by hypothesis, then $d_m(N) = 2(m,N/2)$. 
Using this, and dividing
the sum over $m$ into even and odd terms, we obtain
\begin{equation}\label{gcd sum lemma 0}
\begin{split}
\sum_{m=1}^M W_M(m) \frac{d_m(N)}{m} &= 
2\sum_{1\le m' \le\lfloor M/2\rfloor}W_M(2m')\frac{(2m',N/2)}{2m'} \\ 
&+2\sum_{0\le m' \le\lfloor (M-1)/2\rfloor}W_M(2m'+1)\frac{(2m'+1,N/2)}{2m'+1}. 
\end{split}
\end{equation}
We have $W_M(2m') \le W_{\lceil M/2\rceil}(m')$ and, by definition, $(2m',N/2) =
d_{m'}(N/2)$. It follows by induction that 
the first sum on the r.h.s.\ of \eqref{gcd sum lemma 0} is bounded by
$\tau(N/2)\log \lceil M/2\rceil$.
Furthermore, the second sum is clearly bounded by
\begin{equation}
\sum_{0\le m' \le\lfloor (M-1)/2\rfloor}W_M(2m'+1)\frac{(2m'+1,N')}{m'+1/2}
\le 2(1-1/M)+\sum_{1\le m\le M} W_M(m) \frac{d_m(N')}{m},
\end{equation}
which, by induction, is $\le 2(1-1/M)+ \tau(N')\log M$. 
Therefore, 
using the bound $\log \lceil M/2\rceil \le \log M +1/M -\log 2$
and the formula $\tau(N/2) + \tau(N') = \tau(N)$, we arrive at
\begin{equation}\label{wm eq}
\sum_{m=1}^M W_M(m) \frac{d_m(N)}{m} \le
\tau(N) \log M+(2-2/M +\tau(N/2)/M-\tau(N/2) \log 2).
\end{equation}
We conclude that the bound \eqref{gcd sum eq 1} 
holds provided that
$\tau(N/2)\ge 4$. This is always fulfilled if $a\ge 2$ unless
$N=4$ or $8$. But the lemma follows in these cases also by direct calculation.
\end{proof}

\section{Hybrid van der Corput--Weyl lemmas}

\begin{lemma}\label{corput lemma 1 1}
Suppose that $f$ is a function satisfying the hypothesis of Lemma~\ref{exp reduce}
for some $\lambda>1$, $\eta \ge 0$, and
with $J=1$.
If $f(x)$ is real for real $x$, then
\begin{equation}\label{corput lemma 1 1 result}
\begin{split}
\left|\sum_{n=N+1}^{N+L} \chi(n) e^{2\pi i f(n)} \right| \le &\,  
\frac{2\nu_1(\lambda,\eta) C_1}{\pi}\left(\log \frac{D_1}{2B_1} +\frac{7}{4}+\frac{\pi}{2}\right)\\
&+\frac{\nu_1(\lambda,\eta) C_1}{\pi}\min\left(\frac{\pi B_1 L}{q},
\|qf'(N+1)/B_1\|^{-1}\right).
\end{split}
\end{equation}
\end{lemma}
\begin{proof}
Applying Lemma~\ref{exp reduce} with $J=1$ gives
\begin{equation}\label{temp 1 1 1}
\left|\sum_{n=N+1}^{N+L} \chi(n) e^{2\pi i f(n)} \right| \le  
\nu_1(\lambda,\eta) \max_{0\le \Delta < L} 
\left|\sum_{n=N+1+\Delta}^{N+L} \chi(n)e^{2\pi i
P_1(n-N-1)}\right|,
\end{equation}
where $P_1(x) = f(N+1)+ f'(N+1)x$.
Let $\Delta^*$ be where the maximum is achieved on the r.h.s.\ of 
\eqref{temp 1 1 1}. Let $N^*:=N+\Delta^*$ and
$L^*=L-\Delta^*$. So we have
\begin{equation}\label{corput lemma 1 1 0}
\left|\sum_{n=N+1}^{N+L} \chi(n) e^{2\pi i f(n)} \right| \le  
\nu_1(\lambda,\eta) \left|\sum_{n=N^*+1}^{N^*+L^*} \chi(n)e^{2\pi i
P_1(n-N-1)}\right|.
\end{equation}
We split the range of summation 
$N^*+1\le n\le N^*+L^*$ into arithmetic progressions 
along the residue classes $\ell \pmod{C_1}$. 
For each residue class $0\le \ell < C_1$,
the terms in the progression $n\equiv \ell\pmod {C_1}$
are indexed by the integers $k$ that verify
$N^*+1\le \ell+C_1k \le N^*+L^*$. So we have
$\lceil (N^*+1 -\ell)/C_1\rceil \le k \le \lfloor(N^*+L^*-\ell)/C_1\rfloor$.
Using  the formula $\lceil x+\delta \rceil - \lfloor x\rfloor = 1$,
valid for any $x$ and $\delta \in (0,1)$, 
we deduce that 
$\lceil (N^*+1 -\ell)/C_1\rceil - \lfloor (N^*-\ell)/C_1\rfloor = 1$. 
Therefore, if we define $H_{\ell}:= \lfloor (N^*-\ell)/C_1\rfloor$,
then each $\ell$ determines an integer $\Omega_{\ell} \le \lceil L^*/C_1\rceil$ 
such that (we use the triangle inequality below)
\begin{equation}\label{weyl0 sum 0 0}
\left|\sum_{n=N^*+1}^{N^*+L^*} \chi(n)e^{2\pi i P_1(n-N-1)}\right| \le 
\sum_{\ell =0}^{C_1-1}
\left|\sum_{k=H_{\ell}+1}^{H_{\ell}+\Omega_{\ell}}
\chi(\ell+C_1k)e^{2\pi i P_1(\ell+C_1k-N-1)}\right|.
\end{equation}
From Lemma~\ref{postnikov lemma}, 
and the formula 
$\chi(\ell+ C_1 k) =\chi(\ell)\chi(1+C_1\overline{\ell}k)$, 
valid for $(\ell,q)=1$, 
we deduce that
there are integers $\gamma_1$ and $B_1$ such that
$(\gamma_1,D_1)=1$, $B_1|D_1$, and 
\begin{equation}
\chi(\ell+ C_1 k)=
\chi(\ell)e^{2\pi i B_1\gamma_1 \overline{\ell} k/D_1}, \qquad (\ell,q)=1.
\end{equation}
If $(\ell,q)>1$, then $\chi(\ell+C_1k)=0$. Therefore, 
\begin{equation}\label{weyl0 sum 0}
\left|\sum_{n=N^*+1}^{N^*+L^*} \chi(n)e^{2\pi i P_1(n-N-1)}\right| \le 
\sum_{\substack{\ell =0\\ (\ell,q)=1}}^{C_1-1}
\left|\sum_{k=H_{\ell}+1}^{H_{\ell}+\Omega_{\ell}}
e^{2\pi i (B_1\gamma_1 \overline{\ell} /D_1+ C_1 f'(N+1))k}\right|.
\end{equation}

Let us define 
\begin{equation}
z_f:=\left[\frac{qf'(N+1)}{B_1}\right],\qquad \delta_f:=
\pm \left\|\frac{qf'(N+1)}{B_1}\right\|,
\end{equation}
where $\delta_f$ is positive if 
$z_f$ is obtained by rounding down, and negative if 
$z_f$ is obtained by rounding up. 
In either case, since $D_1C_1=q$ by construction, 
we have $C_1f'(N+1)= (z_f+ \delta_f)B_1/D_1$. Therefore,
\begin{equation}\label{Ul}
\left\| \frac{B_1 \gamma_1\overline{\ell}}{D_1}+C_1 f'(N+1)\right\|
 = \left\| \frac{\gamma_1\overline{\ell} + z_f+\delta_f} {D_1/B_1}\right\|
=:U_{\gamma_1\overline{\ell}+z_f+\delta_f}.
\end{equation}
In view of this, it follows by the Kusmin--Landau Lemma in \cite[Lemma
2]{cheng-graham} that the inner sum in \eqref{weyl0 sum 0} satisfies 
\begin{equation}
\left|\sum_{k=H_{\ell}+1}^{H_{\ell}+\Omega_{\ell}}
e^{2\pi i (B_1\gamma_1 \overline{\ell} /D_1+ C_1 f'(N+1))k}\right|\le 
\min\left(\Omega_{\ell},\frac{1}{\pi} 
U_{\gamma_1\overline{\ell}+z_f+\delta_f}^{-1} +1\right).
\end{equation}
Given this, we divide the sum over $\ell$ in \eqref{weyl0 sum 0} into segments of length
$D_1/B_1$. 
\begin{equation}
\left[\frac{uD_1}{B_1},\frac{(u+1)D_1}{B_1}\right),\quad u\in\mathbb{Z}, 
\quad 0\le u< \frac{B_1 C_1}{D_1}.
\end{equation}
Over each segment, we can 
get an easy handle on $U_{\gamma_1\overline{\ell}+z_f+\delta_f}$. 
Indeed, as $\ell$ runs over the reduced residue classes modulo $q$ 
(hence reduced modulo $D_1/B_1$) in a given segment, 
$\gamma_1 \overline{\ell}+z_f$ runs over a subset of the 
residue classes modulo $D_1/B_1$, hitting each class at most once. 
Therefore, summing over the $B_1C_1/D_1$ segments, and recalling that 
$\Omega_{\ell}\le \lceil L/C_1\rceil$ by construction, 
we obtain  
\begin{equation}\label{corput lemma 1 1 1}
\left|\sum_{n=N^*+1}^{N^*+L^*} \chi(n)e^{2\pi i P_1(n-N-1)}\right|
\le \frac{B_1 C_1}{D_1} \sum_{\ell \pmod*{D_1/B_1}}
\min\left(\lceil L/C_1\rceil,\frac{1}{\pi}U_{\ell+\delta_f}^{-1}+1\right).
\end{equation}
We choose the residue class representatives mod $D_1/B_1$
to be in $[-D_1/2B_1, D_1/2B_1)$
if $\delta_f\ge 0$, 
and in $(-D_1/2B_1, D_1/2B_1]$ if $\delta_f <0$.
In either case, let
$\tilde{\ell}$ denote the representative of $\ell$.
Since $0\le |\delta_f|\le 1/2$, we deduce the formula 
\begin{equation}\label{corput lemma 1 1 2}
U_{\ell+\delta_f} =
\left\{\begin{array}{ll}
\displaystyle
\frac{|\tilde{\ell}|+\sgn(\tilde{\ell})\delta_f}{D_1/B_1} &
\displaystyle
\tilde{\ell} \ne 0,\\ 
&\\
\displaystyle
\frac{|\delta_f|}{D_1/B_1}& 
\displaystyle
\tilde{\ell}=0.
\end{array}\right.
\end{equation}
Now, if $\delta_f\ge 0$, we isolate the terms corresponding to
$\tilde{\ell}=0$ and $\tilde{\ell}=-1$ (if they exist) 
on the r.h.s.\ of \eqref{corput lemma 1 1 1}. 
And if $\delta_f<0$, we isolate the terms for $\tilde{\ell}=0$ and $\tilde{\ell}=1$.
Moreover, we use the lower bound $U_{\pm 1+\delta_f} \ge B_1/2D_1$ to control
the term $\tilde{\ell}=\pm 1$. Then we sum over the remaining $\tilde{\ell}$, 
pairing the terms for 
$\tilde{\ell}$ and $-\tilde{\ell} -1$ if $\delta_f\ge 0$, 
and the terms for $\tilde{\ell}+1$ and $-\tilde{\ell}$
if $\delta_f<0$. In summary, assuming that $D_1/B_1\ge 2$ (so there are at least
two residue class modulo $D_1/B_1$), we obtain
\begin{equation}\label{corput lemma 1 1 3}
\begin{split}
&\sum_{\ell\pmod*{D_1/B_1}}
\min\left(\lceil L/C_1\rceil,\frac{1}{\pi}U_{\ell+\delta_f}^{-1}+1\right)
\le \min\left(\lceil L/C_1\rceil,\frac{1}{\pi}U_{\delta_f}^{-1}+1\right)\\
&+\left(\frac{2D_1}{\pi B_1}+1\right)
+\left(\frac{D_1}{B_1}-2\right)
+\frac{D_1}{\pi B_1}
\sum_{1\le \ell < \frac{D_1}{2B_1}}\left(\frac{1}{\ell+|\delta_f|}
+\frac{1}{\ell+1-|\delta_f|}\right).
\end{split}
\end{equation}
The second sum over $\ell$ on the r.h.s.\ of \eqref{corput lemma 1 1 3} 
is bounded by
\begin{equation}\label{corput lemma 1 1 5}
\sum_{1\le \ell < \frac{D_1}{2B_1}}
\frac{2\ell+1}{\ell^2+\ell+|\delta_f|-\delta_f^2}\le 
\frac{3}{2}+\sum_{2\le \ell < \frac{D_1}{2B_1}} \frac{2}{\ell} \le
\frac{3}{2}+2\log\left(\frac{D_1}{2B_1}\right).
\end{equation}
It is easy to check that the last two estimates still hold if $D_1/B_1=1$.
Hence, substituting \eqref{corput lemma 1 1 5} into \eqref{corput
lemma 1 1 3} 
we obtain, on noting that $\lceil L/C_1\rceil \le L/C_1+1$,
\begin{equation}
\begin{split}
\sum_{\ell\pmod*{D_1/B_1}}
&\min\left(\lceil L/C_1\rceil,\frac{1}{\pi}U_{\ell+\delta_f}^{-1}+1\right)
\le \min\left(\frac{L}{C_1},\frac{1}{\pi} U_{\delta_f}^{-1} \right)\\
&+\frac{2D_1}{\pi B_1}\left(1+\frac{\pi}{2}+\frac{3}{4}\right) 
+\frac{2D_1}{\pi B_1}\log\left(\frac{D_1}{2B_1}\right). 
\end{split}
\end{equation}
We multiply the last estimate by the outside factor $B_1C_1/D_1$ in
\eqref{corput lemma 1 1 1}. This gives
\begin{equation}\label{corput lemma 1 1 6}
\begin{split}
\left|\sum_{n=N^*+1}^{N^*+L^*} \chi(n)e^{2\pi i P_1(n-N-1)}\right|
&\le
\frac{B_1C_1}{D_1}\left(\min\Big(\frac{L}{C_1},\frac{1}{\pi}U_{\delta_f}^{-1}\Big)\right.\\
&+\left.\frac{2C_1}{\pi}\Big(\log
\frac{D_1}{2B_1}+\frac{7}{4}+\frac{\pi}{2}\Big)\right).
\end{split}
\end{equation}
Finally, we use the formula $U_{\delta_f}^{-1} = \|qf'(N+1)/B_1\|^{-1} D_1/B_1$, 
and substitute \eqref{corput lemma 1 1 6} back into \eqref{corput lemma 1 1
0}. After straightforward rearrangements, we obtain the lemma. 
\end{proof}

\begin{lemma}\label{corput lemma 2 2}
Suppose that $f$ is a function satisfying the hypothesis of Lemma~\ref{exp reduce}
for some $\lambda>1$, $\eta \ge 0$, and
with $J=2$. Let $d_m:=(2m,D/B)$.
If $f(x)$ is real for real $x$, then
\begin{equation}\label{corput lemma 2 2 result}
\begin{split}
\left|\sum_{n=N+1}^{N+L} \chi(n) e^{2\pi i f(n)} \right|^2 \le&\,
\frac{4\nu_2(\lambda,\eta)^2\Lambda C L}{\pi}\left(\log
\frac{D}{2B}+\frac{7}{4}+\frac{3\pi}{2\Lambda}\right)\\
&+  \frac{4\nu_2(\lambda,\eta)^2\Lambda
C^2}{\pi}\sum_{m=1}^{\lceil L/C\rceil}\left(1-\frac{m}{\lceil L/C\rceil}\right)
\\
&\times \min\left(\frac{\pi d_m B L}{CD}, 
\left\|\frac{m C^2Df''(N+1)}{Bd_m}\right\|^{-1}\right).
\end{split}
\end{equation}
\end{lemma}

\begin{proof}
We apply Lemma~\ref{exp reduce} with $J=2$ to the sum.
This yields (similarly to the beginning of the 
proof of Lemma~\ref{corput lemma 1 1}) that
\begin{equation}\label{corput lemma 2 2 0}
\left|\sum_{n=N+1}^{N+L} \chi(n) e^{2\pi i f(n)} \right| \le  
\nu_2(\lambda,\eta) \left|\sum_{n=N^*+1}^{N^*+L^*} \chi(n)e^{2\pi i
P_2(n-N-1)}\right|,
\end{equation}
where $P_2(x) = f(N+1) + f'(N+1)x+ f''(N+1)x^2/2$ 
and $[N^*+1,N^*+L^*]\subset [N+1,N+L]$.
We split the range of summation on the r.h.s.\ 
of \eqref{corput lemma 2 2 0} into arithmetic progressions 
along the residue classes $\ell$ modulo $C$. 
Letting $K_{\ell}:= \lfloor (N^*-\ell)/C\rfloor$
and $\Delta_{\ell} :=  \lfloor(N^*+L^*-\ell)/C\rfloor- K_{\ell}$, we have
\begin{equation}\label{corput lemma 2 2 1}
\sum_{n=N^*+1}^{N^*+L^*} \chi(n)e^{2\pi i P_2(n-N-1)} = 
\sum_{\ell =0}^{C-1}\sum_{k=K_{\ell}+1}^{K_{\ell}+\Delta_{\ell}}
\chi(\ell + C k) e^{2\pi i P_2(\ell + C k-N-1)}.
\end{equation}
We make use of the following properties of $\Delta_\ell$. 
First, by construction, we have 
\begin{equation}\label{corput lemma 2 2 2}
\sum_{\ell=0}^{C-1} \Delta_{\ell}=L^* \le L.
\end{equation}
Second, using the periodicity of $\Delta_{\ell}$ 
 as a function of $\ell\pmod C$,
 and the change of variable $r \equiv N^*-\ell \pmod C$,
 we obtain
\begin{equation}\label{corput lemma 2 2 3}
\sum_{\ell=0}^{C-1} \sqrt{\Delta_{\ell}} =\sum_{r=0}^{C-1} 
\sqrt{\lfloor(L^*+r)/C\rfloor}
\le \sum_{r=0}^{C-1} 
\sqrt{\lfloor(L+r)/C\rfloor}.
\end{equation}
Furthermore, supposing that $L\equiv \ell_0 \pmod C$, where $0\le \ell_0< C$, 
we obtain on considering the summation ranges 
$0\le r\le C-\ell_0-1$ and $C-\ell_0\le r \le C-1$ 
in \eqref{corput lemma 2 2 3} separately that 
\begin{equation}
\sum_{r=0}^{C-1} \sqrt{\lfloor(L+r)/C\rfloor} = (C-\ell_0) \sqrt{(L-\ell_0)/C}
+ \ell_0 \sqrt{(L-\ell_0)/C+1}.
\end{equation}
If we view the r.h.s.\ above as a function of $0\le \ell_0 < C$, say $p(\ell_0)$, 
then its maximum is achieved when $\ell_0=0$. Thus, 
\begin{equation}\label{corput lemma 2 2 4}
\sum_{\ell=0}^{C-1} \sqrt{\Delta_{\ell}} \le p(0)=\sqrt{CL}.
\end{equation}
Also, we have the bound
\begin{equation}\label{corput lemma 2 2 4 0}
\sum_{\ell=0}^{C-1} \Delta_{\ell}^2 \le
\frac{L^2}{C}+(\tilde{\rho}-\tilde{\rho}^2)C,
\qquad \tilde{\rho}:=\ell_0/C.
\end{equation}

We are now ready to return to \eqref{corput lemma 2 2 1}.
Lemma~\ref{postnikov lemma} 
asserts that there is a  polynomial $g_{\ell}(x)$ of degree $2$ in $x$ such that 
\begin{equation}\label{corput lemma 2 2 5}
\chi(\ell+C k) = \chi(\ell)e^{2\pi i g_{\ell}(k)}, \qquad (\ell,q)=1,
\end{equation}
where $g_{\ell}(x) = \alpha_{\ell}x +B \gamma \overline{\ell}^2 x^2/D$,
$(\gamma,q)=1$, and $B|D$.
Therefore, applying the Cauchy--Schwarz inequality to the r.h.s.\ in \eqref{corput lemma 2 2 1},
we obtain
\begin{equation}\label{corput lemma 2 2 6}
\left|\sum_{n=N^*+1}^{N^*+L^*} \chi(n)e^{2\pi i P_2(n-N-1)}\right|^2 \le 
C\sum_{\substack{\ell=0\\ (\ell,q)=1}}^{C-1}\left|\sum_{k=K_{\ell}+1}^{K_{\ell}+\Delta_{\ell}}
 e^{2\pi i Q_{\ell}(k)}\right|^2.
\end{equation}
where $Q_{\ell}(x) := g_{\ell}(x)+P_2(\ell + C x-N-1)$.
We bound the inner sum using the van der Corput--Weyl Lemma
in \cite[Lemma 5]{cheng-graham}.
In fact, we use the more precise form
of the lemma at the bottom of
 page 1273. 
This form implies that if $M$ is a positive integer then 
\begin{equation}\label{corput lemma 2 2 7}
\begin{split}
\Big|\sum_{k=K_{\ell}+1}^{K_{\ell}+\Delta_{\ell}} e^{2\pi i Q_{\ell}(k)} \Big|^2
\le &  (\Delta_{\ell}+M)\Big( 
\frac{\Delta_{\ell}}{M}+\frac{2}{M}\sum_{m=1}^{M}
\left(1-\frac{m}{M}\right)|S_m'(\ell)|\Big),
\end{split}
\end{equation}
where
\begin{equation}\label{corput lemma 2 2 8}
S_m'(\ell):= \sum_{r = K_{\ell}+1}^{K_{\ell}+\Delta_{\ell}-m} e^{2\pi
i (Q_{\ell}(r+m)- Q_{\ell}(r))}.
\end{equation}
Substituting \eqref{corput lemma 2 2 7} into \eqref{corput lemma 2 2 6}, 
and using the properties \eqref{corput lemma 2 2 2}
and \eqref{corput lemma 2 2 4 0}, 
together with the upper 
bound $\Delta_{\ell}\le \lceil L/C\rceil$, we obtain
\begin{equation}\label{corput lemma 2 2 9}
\begin{split}
\Big|\sum_{n=N^*+1}^{N^*+L^*} \chi(n)e^{2\pi i P_2(n-N-1)}\Big|^2 &\le 
CL + \frac{L^2+\tilde{\rho}(1-\tilde{\rho})C^2}{M}\\
&+2C \left(1+\frac{\lceil L/C\rceil}{M}\right)
\sum_{m=1}^{M}\left(1-\frac{m}{M}\right) 
\sum_{\substack{\ell=0\\ (\ell,q)=1}}^{C-1}|S_m'(\ell)|.
\end{split}
\end{equation}
Since
$Q_{\ell}(x)$ is a quadratic polynomial, we have the simpler
expression
\begin{equation}
|S'_m(\ell)| = \left|\sum_{r=K_{\ell}+1}^{K_{\ell}+\Delta_{\ell}-m}
e^{2\pi i (2mB \gamma\overline{\ell}^2/D + mC^2 f''(N+1)) r)}\right|.
\end{equation}
We plan to bound $S'_m(\ell)$ using the Kusmin--Landau Lemma in \cite[Lemma
2]{cheng-graham}. With this in mind, recall the definition  $d_m=(2m,D/B)$. 
Let
\begin{equation}
m'=:\frac{2m}{d_m},
\qquad P_m:=\frac{D}{Bd_m}.
\end{equation}
Let\footnote{If each prime factor of $q$ occurs with multiplicity $>1$, then 
$C^2D= q$ if $q$ is odd, and $C^2D= 2q$ if $q$ is even. So the expressions that
follow can be simplified in this case.}
\begin{equation}
\begin{split}
&w_m:= [P_m m C^2f''(N+1)] = \left[\frac{m' C^2D f''(N+1)}{2B}\right],\\
&\epsilon_m:=\|P_mmC^2f''(N+1)\| =\pm \left\|\frac{m' C^2D f''(N+1)}{2B}\right\|.
\end{split}
\end{equation}
Here, $\epsilon_m$ is positive if            
 $w_m$ is obtained by rounding down, and negative if
 $w_m$ is obtained by rounding up. Hence,
\begin{equation}
 \left\| \frac{2 m B \gamma \overline{\ell}^2}{D} + m C^2 f''(N+1)\right\|
 =  \left\| \frac{m'\gamma \overline{\ell}^2+ w_m
 +\epsilon_m}{P_m}\right\|.
\end{equation}
Letting $U_{z,m} := \| z/P_m\|$, 
the Kusmin--Landau Lemma furnishes the estimate
\begin{equation}\label{corput lemma 2 2 11}
|S'_m(\ell)| \le 
\min\left(\Delta_{\ell}-m,\frac{1}{\pi}
U_{m'\gamma\overline{\ell}^2+w_m+\epsilon_m,m}^{-1}+1\right).
\end{equation}
Therefore, using the inequality $\Delta_{\ell}\le \lceil L/C\rceil$
yields
\begin{equation}
\mathcal{S}_m=
\sum_{\substack{\ell=0\\ (\ell,q)=1}}^{C-1}|S_m'(\ell)|
\le \sum_{\substack{\ell=0\\ (\ell,q)=1}}^{C-1}
\min\left(\lceil L/C\rceil-m,\frac{1}{\pi}
U_{m'\gamma\overline{\ell}^2+w_m+\epsilon_m,m}^{-1}+1\right).
\end{equation}
To get an explicit expression for $U_{z,m}$, 
we consider subsums of $\mathcal{S}_m$ 
over the segments 
\begin{equation}
\left[uP_m,(u+1)P_m\right),\quad u\in\mathbb{Z}, \quad
0\le u< C/P_m.
\end{equation}
To this end, let
\begin{equation}
\Lambda_m:=\#\{0\le x<P_m\,:\, x^2\equiv 1\pmod{P_m}\}.
\end{equation}
As $\ell$ runs over the reduced residue classes in each segment,  
then since $(m'\gamma,P_m)=1$ and $\overline{\ell}$ is squared,
if $m'\gamma\overline{\ell}^2+w_m$ hits
a residue class modulo $P_m$, it does so $\Lambda_m$ times.
Let $\mathcal{R}_m$ denote the classes that are hit.
We have that the cardinality of $\mathcal{R}_m$ is 
$\le P_m/\Lambda_m$. 
If $\epsilon_m\ge 0$
we take  $\mathcal{R}_m\subset [-P_m/2, P_m/2)$ 
while if $\epsilon_m<0$ we take 
$\mathcal{R}_m\subset (-P_m/2, P_m/2]$.  
Furthermore, given $m'\gamma \bar{\ell}^2+w_m$, 
let $\tilde{\ell}\in \mathcal{R}_m$ denote the class representative 
that it hits.
Then, summing over the $C/P_m$ segments, we obtain
\begin{equation}
\mathcal{S}_m \le \frac{C\Lambda_m}{P_m}
\sum_{\tilde{\ell}\in\mathcal{R}_m}\min\left(\lceil L/C\rceil-m,\frac{1}{\pi}
U^{-1}_{\tilde{\ell}+\epsilon_m,m}+1\right),
\end{equation}
and we have the formula
\begin{equation}
U_{\tilde{\ell}+\epsilon_m,m} = \left\{\begin{array}{ll}
\displaystyle
\frac{|\tilde{\ell}|+\sgn(\tilde{\ell})\epsilon_m}{P_m}&
\displaystyle \tilde{\ell}\ne 0,\\
\,&\\
\displaystyle
\frac{|\epsilon_m|}{P_m}& \tilde{\ell} = 0.
\end{array}\right.
\end{equation}
At worst, the classes that are hit concentrate in
$[-P_m/2\Lambda_m,P_m/2\Lambda_m]$. If $\epsilon_m\ge 0$, 
we isolate the terms corresponding to $\tilde{\ell}=0$ and $\tilde{\ell}=-1$ (if they exist),
and pair the remaining terms for $\tilde{\ell}$ and $-\tilde{\ell}-1$. 
While if $\epsilon_m<0$, we isolate the terms for $\tilde{\ell} =0$ and
$\tilde{\ell} = 1$,
and pair the remaining terms for $\tilde{\ell}+1$ and $-\tilde{\ell}$. 
Since $0\le |\epsilon_m|\le 1/2$ and $P_m/\Lambda_m\ge 1$, 
this gives 
\begin{equation}
\begin{split}
\mathcal{S}_m &\le  \frac{C\Lambda_m}{P_m}\min\left(\lceil L/C\rceil-m,\frac{P_m}{\pi
|\epsilon_m|}+1\right)+  \frac{C\Lambda_m}{P_m}\left(\frac{2P_m}{\pi}+1\right)\\
&+  \frac{C\Lambda_m}{P_m}\left(\frac{P_m}{\Lambda_m}-1\right) 
+\frac{C \Lambda_m}{\pi}
\sum_{1\le \ell< \frac{P_m}{2\Lambda_m}} 
\left(\frac{1}{\ell+|\epsilon_m|}+\frac{1}{\ell+1-|\epsilon_m|}\right).
\end{split}
\end{equation}
Furthermore, 
\begin{equation}
\sum_{1\le \ell< \frac{P_m}{2\Lambda_m}} 
\left(\frac{1}{\ell+|\epsilon_m|}+\frac{1}{\ell+1-|\epsilon_m|}\right)
\le \sum_{1\le \ell< \frac{P_m}{2\Lambda_m}}\frac{2\ell+1}{\ell^2+\ell}
\le \frac{3}{2}+ 2\log \frac{P_m}{2\Lambda_m}.
\end{equation}
Hence, using to the trivial inequalities $\lceil L/C\rceil < L/C+1$ 
and $1\le \Lambda_m\le P_m$, together with the observation $P_m|D$ so that 
$\Lambda_m=\Lambda(P_m) \le \Lambda(D)= \Lambda$, we obtain
\begin{equation}
\mathcal{S}_m \le 
\frac{C\Lambda_m}{P_m}\min\left(
L/C-m,\frac{P_m}{\pi |\epsilon_m|}\right)
+\frac{2C\Lambda}{\pi}+2C
+\frac{C\Lambda}{\pi}\left(\frac{3}{2}+ 2\log \frac{D}{2B}\right).
\end{equation}
Now, we have $\sum_{1\le m\le M}(1-m/M) = (M-1)/2$.
So summing over $m$ we obtain
\begin{equation}\label{corput lemma 2 2 12}
\begin{split}
\sum_{m=1}^{M}\left(1-\frac{m}{M}\right) \mathcal{S}_m 
&\le C\sum_{m=1}^{M}\left(1-\frac{m}{M}\right)\frac{\Lambda_m}{P_m}
\min\left( L/C-m,\frac{P_m}{\pi
|\epsilon_m|}\right)\\
&+C(M-1)+
\frac{C\Lambda (M-1)}{\pi}\left(\frac{7}{4}+\log \frac{D}{2B}\right).
\end{split}
\end{equation}
In \eqref{corput lemma 2 2 12}, 
we choose $M=\lceil L/C\rceil$, so that $M=L/C+1-\tilde{\rho}$. 
Then we substitute the
resulting expression into \eqref{corput lemma 2 2 9}, which gives
\begin{equation}\label{corput lemma 2 2 14}
\begin{split}
\Big|\sum_{n=N^*+1}^{N^*+L^*} \chi(n)e^{2\pi i P_2(n-N-1)}\Big|^2 &\le 
CL+\frac{L^2+\tilde{\rho}(1-\tilde{\rho})C^2}{L/C+1-\tilde{\rho}}\\
&+4C^2\sum_{m=1}^{\lceil L/C\rceil} \left(1-\frac{m}{\lceil L/C\rceil}\right)\frac{\Lambda_m}{P_m}
\min\left(L/C-m,\frac{P_m}{\pi
|\epsilon_m|}\right)\\
&+4CL
+\frac{4\Lambda C L}{\pi}\left(\frac{7}{4}+\log \frac{D}{2 B}\right).
\end{split}
\end{equation}
At this point, we may assume that $L \ge C$, otherwise the lemma 
is trivial due to the first term in \eqref{corput lemma 2 2 result}.
Given this assumption, it is easy to check that the second term in
\eqref{corput lemma 2 2 14}, viewed as a function of $\tilde{\rho}$, has no critical
points in the interval $[0,1)$, and so it is monotonic over that interval. 
Comparing the values at $\tilde{\rho}=0$ and $\tilde{\rho}=1$, we deduce that the maximum
is at $\tilde{\rho}=1$. 
Using this in \eqref{corput lemma 2 2 14}, and 
substituting the result into \eqref{corput lemma 2 2 0} (after squaring both
sides there) yields the lemma.
\end{proof}

\section{Proof of Theorem~\ref{main theorem}}

If $\chi=\chi_0$ is the principal character, then
\begin{equation}
L(s,\chi_0) = \zeta(s)\prod_{p|q} (1-p^{-s}).
\end{equation}
Bounding the product above trivially, we obtain
\begin{equation}
\begin{split}
|L(1/2+it,\chi_0)|&\le 
|\zeta(1/2+it)|
\prod_{p|q}\left(1+1/\sqrt{p}\right) \le |\zeta(1/2+it)| \tau(q).
\end{split}
\end{equation}
(Note that this is a large overestimate, but it is still fine since the
difficult part of the proof is $\chi\ne \chi_0$.)
Combining this with the bound 
for the Riemann zeta function in \cite{hiary-corput}, we arrive at
\begin{equation}\label{zeta bound}
|L(1/2+it,\chi_0)| \le 0.63 \tau(q) \mathfrak{q}^{1/6}\log\mathfrak{q},\qquad
(|t| \ge 3).
\end{equation}
So the theorem follows in this case.
Henceforth, we assume 
that $\chi$ is nonprincipal, and so $q>2$.

Let $\rho =1.3$, which is 
a parameter that will control the size of the
segments in our dyadic subdivision. 
The starting point of the dyadic subdivision is 
\begin{equation}
v_0 =\left\lceil \frac{C |t|^{1/3}}{(\rho-1)^2}\right\rceil.
\end{equation}
We assume that $|t|\ge t_0\ge \rho^3/(\rho-1)^3$ 
where $t_0:=200$.
Since $q>2$ by assumption, then $\mathfrak{q}\ge \mathfrak{q}_0:= 3t_0$. 

From the Dirichlet series definition of $L(s,\chi)$, we have
\begin{equation}\label{trivial dirichlet bound}
|L(1/2+it,\chi)| \le \left|\sum_{n=1}^{\infty} \frac{\chi(n)}{n^{1/2+it}}\right|.
\end{equation}
We divide the summation range on the r.h.s.\ of \eqref{trivial dirichlet bound}
into an initial sum followed by dyadic segments 
$[\rho^{\ell} v_0,\rho^{\ell+1}v_0)$. 
This gives 
\begin{equation}
|L(1/2+it,\chi)| \le \left|\sum_{n=1}^{v_0-1} \frac{\chi(n)}{n^{1/2+it}}\right|
+\sum_{\ell=0}^{\infty} \left|\sum_{\rho^{\ell} v_0\le n<\rho^{\ell+1}v_0}
\frac{\chi(n)}{n^{1/2+it}}\right|.
\end{equation}
The $\ell$-th dyadic segment is subdivided into blocks of length $L_{\ell}$
where
\begin{equation}
L_{\ell}=\left\{\begin{array}{ll}
\left\lceil \frac{(\rho-1)\rho^{\ell}v_0}{|t|^{1/3}}
\right\rceil, & 0\le \ell< \ell_0:=
\frac{\log(C D|t|^{2/3}/v_0)}{\log\rho}\\
&\\
\left\lceil \frac{(\rho-1)\rho^{\ell} v_0}{|t|^{1/2}}\right\rceil, & 
\ell_0 \le \ell < 
\ell_1:=\frac{\log(q|t|/5 v_0)}{\log \rho},\\
&\\
\left\lceil \frac{(\rho-1)\rho^{\ell} v_0}{|t|}\right\rceil, & 
\ell_1 \le \ell,
\end{array}\right.
\end{equation}
plus a (possibly empty) boundary block. (Note that our assumption $|t|\ge t_0$ and
 the fact $CD\le q$ imply that $\ell_0<\ell_1$.) So there are 
$R_{\ell} = \lceil (\rho-1)\rho^{\ell}v_0/L_{\ell}\rceil$
blocks in the $\ell$-th segment. 
The $r$-th block
in the $\ell$-th segment 
begins at
\begin{equation}
N_{r,\ell}+1= \lceil \rho^{\ell} v_0\rceil +rL_{\ell}, \qquad 
(0\le r< R_{\ell}).
\end{equation}
We first bound the initial sum, then we bound 
the sum over each range of $\ell$ separately.

\subsection{Initial sum}
The initial sum is bounded trivially using the triangle inequality,
the fact $|\chi(n)/n^{1/2+it}|\le 1/\sqrt{n}$, and on comparing with the integral
$\int_0^{v_0-1} \frac{1}{\sqrt{x}}dx$. Recalling that $C/q^{1/3}=\cbf(q)$, this
gives
\begin{equation}\label{main theorem region 0}
 \left|\sum_{n=1}^{v_0-1} \frac{\chi(n)}{n^{1/2+it}}\right| 
 \le 2\sqrt{v_0-1} \le \mathfrak{v}_0 \sqrt{\cbf(q)}\mathfrak{q}^{1/6},
\qquad
\mathfrak{v}_0:= \frac{2}{\rho-1}.
\end{equation}

\subsection{Sum over $0\le \ell < \ell_0$}
Using the Cauchy--Schwarz inequality we obtain
\begin{equation}\label{main theorem 0}
\left|\sum_{0\le \ell<\ell_0} \sum_{\rho^{\ell} v_0\le n<\rho^{\ell+1}v_0}
\frac{\chi(n)}{n^{1/2+it}}\right|^2\le
(\ell_0+1)\sum_{0\le \ell<\ell_0} 
\left|\sum_{\rho^{\ell} v_0\le n<\rho^{\ell+1}v_0}
\frac{\chi(n)}{n^{1/2+it}}\right|^2
\end{equation}
We partition
the $\ell$-th dyadic segment in \eqref{main theorem 0}
into blocks of length $L_{\ell}$.
Then we apply partial summation to each segment 
to remove the weighting factor $1/\sqrt{n}$.
Finally, we apply the Cauchy--Schwarz inequality to the sum of the blocks. 
This yields
\begin{equation}\label{main theorem 1}
\begin{split}
&\left|\sum_{\rho^{\ell} v_0\le n < \rho^{{\ell}+1}v_0} 
\frac{\chi(n)}{n^{1/2+it}}\right|^2 \le 
\frac{R_{\ell}}{\rho^{\ell}v_0} 
\sum_{r=0}^{R_{\ell}-1} \max_{0\le \Delta \le L_{\ell}}\left| 
\sum_{n=N_{r,\ell}+1}^{N_{r,\ell}+ \Delta} 
\frac{\chi(n)}{n^{it}}\right|^2.
\end{split}
\end{equation}
We estimate the inner sum in \eqref{main theorem 1} 
via Lemma~\ref{corput lemma 2 2}.
To this end, let 
\begin{equation}
\lambda = \frac{1}{\sqrt{\rho-1}},\qquad 
f(x)=-\frac{t}{2\pi}\log x,
\end{equation}
and $1\le L=\Delta\le L_{\ell}$. (Note that $\lambda>1$, as required by the
lemma.)
We have
\begin{equation}\label{analyticity condition}
\frac{\lambda (L_{\ell}-1)}{N_{r,\ell}+1} < \frac{1}{\lambda |t|^{1/3}}< 1.
\end{equation}
So $f(N_{r,\ell}+1+z)$ is analytic on a disk of radius $|z|\le
\lambda(L_{\ell}-1)$. 
Moreover, as a consequence of \eqref{analyticity
condition},
\begin{equation}
\frac{|f^{(j)}(N_{r,\ell}+1)|}{j!}\lambda^j (L-1)^j = 
\frac{|t|\lambda^j(L-1)^j}{2\pi j (N_{r,\ell}+1)^j}
\le \frac{1}{2\pi j \lambda^j},\qquad (j\ge 3).
\end{equation}
In particular, the required bound on $|f^{(j)}(N_{r,\ell}+1)|/j!$ 
in Lemma~\ref{corput lemma 2 2} holds with $\eta = 1/6\pi$.
Therefore, letting 
$\nu_2=\nu_2(1/\sqrt{\rho-1},1/6\pi)$ and 
\begin{equation}
y_{r,m,\ell} := \frac{mC^2D f''(N_{\ell,r}+1)}{Bd_m}
= \frac{m}{d_m} \frac{C^2 D t}{2\pi B}\frac{1}{(N_{r,\ell}+1)^2},
\end{equation}
Lemma~\ref{corput lemma 2 2} gives 
that the r.h.s.\ in \eqref{main theorem 0} is bounded by 
\begin{equation}\label{main theorem 2}
\frac{4\nu_2^2\Lambda}{\pi} (\ell_0+1)\sum_{0\le \ell <\ell_0} \left(*_{\ell}+**_{\ell}\right),
\end{equation}
where
\begin{equation}
\begin{split}
&*_{\ell} :=\frac{CL_{\ell}R_{\ell}^2}{\rho^{\ell}v_0}
\left(\log\frac{D}{2B}+\frac{7}{4}+\frac{3\pi}{2\Lambda}\right)\\
&\textrm{$**_{\ell}$}:=\frac{C^2R_{\ell}}{\rho^{\ell} v_0}
\sum_{m=1}^{\lceil L_{\ell}/C\rceil}W(m)
\sum_{r=0}^{R_{\ell}-1} \min\left(\frac{\pi d_m B L_{\ell}}{CD},
\frac{1}{\|y_{r,m,\ell}\|}\right),
\end{split}
\end{equation}
where for brevity we write 
\begin{equation}
W(m)=1-\frac{m}{\lceil L_{\ell}/C\rceil}.
\end{equation}
We consider the term $*_{\ell}$ first since it is easier to handle. 
Since $(\rho-1)^2v_0/|t|^{1/3} \ge  C$, we obtain
\begin{equation}\label{Lell bound0}
\frac{(\rho-1)\rho^{\ell}v_0}{|t|^{1/3}}
\le L_{\ell}\le
\frac{(\rho-1)\rho^{\ell+1}v_0}{|t|^{1/3}}.
\end{equation}
And the upper bound in \eqref{Lell bound0} gives
$(\rho-1)^2 v_0/L_{\ell} \ge 1$. Hence,
\begin{equation}\label{Rell bound}
\frac{(\rho-1)\rho^{\ell}v_0}{L_{\ell}} \le R_{\ell} \le
\frac{(\rho-1)\rho^{\ell+1}v_0}{L_{\ell}}.
\end{equation}
As can be seen from \eqref{Rell bound} and the definition of $L_{\ell}$, we have 
\begin{equation}\label{Rell bound 2}
R_{\ell} \le \rho |t|^{1/3}, \qquad (0\le \ell < \ell_0). 
\end{equation}
Using this bound, the bound \eqref{Rell bound}, 
and the inequality $\Lambda \ge 2$ (valid since $q>2$ by assumption),
we arrive at
\begin{equation}\label{main theorem 4}
\sum_{0\le \ell < \ell_0}*_{\ell} \le 
(\ell_0+1) \rho^2(\rho-1) C |t|^{1/3}
\left(\log \frac{D}{2}+\frac{7}{4}+\frac{3\pi}{4}\right).
\end{equation}
Furthermore, by our choice of $v_0$, we have
\begin{equation}\label{ell0 bound}
\ell_0+1 \le 
\frac{\log\left(\rho(\rho-1)^2 D|t|^{1/3}\right)}{\log\rho}.
\end{equation}
Therefore, using
the inequality $\log\frac{D|t|^{1/3}}{2} \le
\log\mathfrak{q}^{1/3}$, 
the formula $C/q^{1/3} = \cbf(q)$,
and incorporating the additional factor $\ell_0+1$ from 
\eqref{main theorem 2} into our estimate, gives
\begin{equation}\label{star eq 1}
(\ell_0+1) \sum_{0\le \ell<\ell_0} *_{\ell}\le \mathfrak{v}_1 \cbf(q) \mathfrak{q}^{1/3}
Z_1(\log \mathfrak{q}),\qquad
\mathfrak{v}_1:=\frac{\rho^2(\rho-1)}{27\log^2 \rho},
\end{equation}
where
\begin{equation}
Z_1(X) := \left(X-\log t_0 +21/4+ 9\pi/4\right)
\left(X+3\log(2\rho(\rho-1)^2)\right)^2.
\end{equation}

The term $**_{\ell}$ in \eqref{main theorem 2} is more 
complicated to handle. 
First, we apply Lemma~\ref{well spacing lemma}
to estimate the sum over $r$ there.
To this end, note that
\begin{equation}
(N_{r+1,\ell}+1)^2-(N_{r,\ell}+1)^2\ge 2\lceil \rho^{\ell} v_0\rceil
L_{\ell}, \qquad  (0\le r < R_{\ell}-1).
\end{equation}
Moreover, by construction,
$N_{R_{\ell}-1,\ell}+1\le \lfloor \rho^{\ell+1} v_0\rfloor$ and 
$N_{0,\ell} +1\ge \lceil \rho^{\ell}v_0\rceil$.
Hence,
\begin{equation}\label{well spacing parameters}
\begin{split}
&|y_{r+1,m,\ell}-y_{r,m,\ell}| 
\ge\frac{m}{d_m} \frac{C^2D|t|}{2\pi B}
\frac{2\lceil \rho^{\ell} v_0\rceil L_{\ell}}{\lfloor
\rho^{\ell+1}v_0\rfloor^4},\qquad (0\le r<R_{\ell}-1),\\
&|y_{R_{\ell}-1,m,\ell}-y_{0,m,\ell}|  
\le \frac{m}{d_m} \frac{C^2D|t|}{2\pi B}
\frac{\rho^2-1}{\lfloor \rho^{\ell+1}v_0\rfloor^2}.
\end{split}
\end{equation}

We apply Lemma~\ref{well spacing lemma} to the sequence $\{y_{r,m,\ell}\}_r$ with
$y=y_{R_{\ell}-1,m,\ell}$, $x=y_{0,m,\ell}$, 
$P=\pi d_m B L_{\ell}/CD$, 
 and (since $y_{r,m,\ell}$ is monotonic in $r$) 
with $\delta$ equal to 
the lower bound for $|y_{r+1,m,\ell}-y_{r,m,\ell}|$
in \eqref{well spacing parameters}.
Using these parameter choices, 
Lemma~\ref{well spacing lemma} gives  
\begin{equation}
\sum_{r=0}^{R_{\ell}-1} \min\left(\frac{\pi d_m B L_{\ell}}{CD},
\frac{1}{\|y_{r,m,\ell}\|}\right)
\le 2(y-x+1)(2P+\delta^{-1}\log (e\max\{P,2\}/2)). 
\end{equation}
Multiplying out the brackets,
we obtain four terms: $2(y-x)\delta^{-1} \log (e\max\{P,2\}/2))$,
$2\delta^{-1}\log (e\max\{P,2\}/2))$, $4(y-x)P$, and $4P$. 
We estimate the sum of each term over $m$ with the aid of
the following
inequalities, which are either straightforward to prove (left two inequalities) or are a consequence of
Lemma~\ref{gcd sum lemma}.
\begin{equation}\label{m bounds}
\begin{split}
&\sum_{m=1}^{\lceil L_{\ell}/C\rceil}W(m) \le \frac{L_{\ell}}{2C},
\qquad
\sum_{m=1}^{\lceil L_{\ell}/C\rceil} W(m)\frac{d_m}{m}
\le  \tau\left(D/B\right)\log \lceil L_{\ell}/C\rceil,\\
&\sum_{m=1}^{\lceil L_{\ell}/C\rceil}
W(m) m \le  \frac{L_{\ell}^2}{2C^2},
\qquad
\sum_{m=1}^{\lceil L_{\ell}/C\rceil}
W(m)d_m
\le \tau\left(D/B\right) \lceil L_{\ell}/C\rceil.
\end{split}
\end{equation}
Combining 
\eqref{Lell bound0}, \eqref{Rell bound 2}, 
\eqref{well spacing parameters}, 
and  \eqref{m bounds}, 
together with the bound (here we use $L_{\ell}\ge C$)
\begin{equation}
\log\left(e\max\{P,2\}/2\right) \le 
\log\left(\frac{e\pi L_{\ell}}{2C}\right),
\end{equation}
we routinely deduce the estimates
\begin{equation}\label{main theorem 5}
\begin{split}
&\frac{C^2R_{\ell}}{\rho^{\ell}v_0}
\sum_{m=1}^{\lceil L_{\ell}/C\rceil}
W(m) \frac{(\rho^2-1)\lfloor \rho^{\ell+1}v_0 \rfloor^2}{ \lceil
\rho^{\ell}v_0\rceil L_{\ell}}
\log\left(\frac{e\pi L_{\ell}}{2C}\right)
\le \mathcal{B}_1(\ell)\\
&\frac{C^2R_{\ell}}{\rho^{\ell}v_0}
\sum_{m=1}^{\lceil L_{\ell}/C\rceil}
W(m) \frac{d_m}{m}
\frac{2\pi B}{C^2D|t|}
\frac{\lfloor \rho^{\ell+1}v_0\rfloor^4}
{\lceil \rho^{\ell} v_0\rceil L_{\ell}}
\log\left(\frac{e\pi L_{\ell}}{2C}\right)
\le \mathcal{B}_2(\ell)\\ 
& \frac{C^2R_{\ell}}{ \rho^{\ell}v_0}
\sum_{m=1}^{\lceil L_{\ell}/C\rceil}
W(m)\frac{4\pi d_m B L_{\ell}}{CD}
\frac{m}{d_m} \frac{C^2D|t|}{2\pi B}
\frac{\rho^2-1}{\lfloor \rho^{\ell+1}v_0\rfloor^2}
\le \mathcal{B}_3(\ell)\\
&\frac{C^2R_{\ell}}{\rho^{\ell}v_0}
\sum_{m=1}^{\lceil L_{\ell}/C\rceil}
W(m) d_m \frac{4\pi B L_{\ell}}{CD}
\le \mathcal{B}_4(\ell),
\end{split}
\end{equation}
where
\begin{equation}
\begin{split}
&\mathcal{B}_1(\ell):=\frac{\rho^3(\rho^2-1)}{2}C|t|^{1/3} \log\left(
\frac{e\pi L_{\ell}}{2C}\right),\\ 
&\mathcal{B}_2(\ell):=
\frac{2\pi\rho^5}{(\rho-1)}\frac{\rho^{\ell}v_0}{D|t|^{1/3}} 
B\tau(D/B)\log\left(\frac{eL_{\ell}}{C}\right)
\log\left(\frac{e\pi L_{\ell}}{2C}\right),\\ 
&\mathcal{B}_3(\ell):=\rho^3(\rho-1)^3(\rho^2-1) C|t|^{1/3}, \\
&\mathcal{B}_4(\ell):=4\pi \rho^2(\rho-1)^2\frac{\rho^{\ell}v_0}{D|t|^{1/3}}
B\tau(D/B).
\end{split}
\end{equation}
Incorporating the additional factor $\ell_0+1$ from
\eqref{main theorem 0} into our estimate, we therefore conclude that
\begin{equation}\label{Bcal bound}
(\ell_0+1)\sum_{0\le \ell<\ell_0} **_{\ell} \le \sum_{j=1}^4(\ell_0+1)\sum_{0\le \ell<\ell_0} \mathcal{B}_j(\ell).
\end{equation}
We estimate the sum over $\ell$ in \eqref{Bcal bound} as a geometric progression.
To this end, we use the bound on $\ell_0$ in \eqref{ell0 bound}, 
the bound
\begin{equation}\label{Lell bound}
L_{\ell} \le \rho(\rho-1)CD|t|^{1/3},\qquad (0\le \ell < \ell_0).
\end{equation}
which follows directly from the definitions of $L_{\ell}$ and $\ell_0$, 
and consequently the bound
\begin{equation}
\log\left(\frac{e\pi L_{\ell}}{2C}\right) 
\le \log\left(\frac{e\pi \rho(\rho-1) D|t|^{1/3}}{2}\right),
\qquad (0\le \ell<\ell_0).
\end{equation}
After some rearrangements, this yields
\begin{equation}\label{main theorem 3}
\begin{split}
&(\ell_0+1)\sum_{0\le \ell < \ell_0} B_1(\ell) \le 
\mathfrak{v}_2 \cbf(q) \mathfrak{q}^{1/3}Z_2(\log \mathfrak{q}),\qquad
\mathfrak{v}_2:= \frac{\rho^3(\rho^2-1)}{54\log^2 \rho},\\
&(\ell_0+1)\sum_{0\le \ell < \ell_0} B_2(\ell)  \le 
\mathfrak{v}_3 \cbf(q) B\tau(D/B) \mathfrak{q}^{1/3}Z_3(\log\mathfrak{q}),\qquad
\mathfrak{v}_3:= \frac{2\pi \rho^6}{27(\rho-1)^2\log \rho},\\
& (\ell_0+1)\sum_{0\le \ell < \ell_0} B_3(\ell)  \le 
\mathfrak{v}_4 \cbf(q) \mathfrak{q}^{1/3}Z_4(\log\mathfrak{q}),\qquad
\mathfrak{v}_4:= \frac{\rho^3(\rho-1)^3(\rho^2-1)}{9\log^2
\rho},\\
& (\ell_0+1)\sum_{0\le \ell < \ell_0} B_4(\ell)  \le
\mathfrak{v}_5\cbf(q) B\tau(D/B) \mathfrak{q}^{1/3}Z_5(\log\mathfrak{q}),\qquad
\mathfrak{v}_5:= \frac{4\pi
\rho^3(\rho-1)}{3\log\rho}.
\end{split}
\end{equation}
where
\begin{equation}
\begin{split}
&Z_2(X):= \left(X+ 3\log(e\pi \rho(\rho-1)\right)\left(X+3\log(2\rho(\rho-1)^2)\right)^2   \\
&Z_3(X):= \left(X+ 3\log(2e\rho(\rho-1)\right)   
\left(X+ 3\log(e\pi \rho(\rho-1)\right)\times \\
&\qquad\qquad \left(X+3\log(2\rho(\rho-1)^2)\right)\\
&Z_4(X):= \left(X+3\log(2\rho(\rho-1)^2)\right)^2   \\
&Z_5(X):= X+3\log(2\rho(\rho-1)^2)   \\
\end{split}
\end{equation}
Therefore,
\begin{equation}\label{star eq 2}
\begin{split}
(\ell_0+1)\sum_{0\le \ell <\ell_0} **_{\ell} &\le
\mathfrak{v}_2\cbf(q)\mathfrak{q}^{1/3}Z_2(\log \mathfrak{q})\\
&+\mathfrak{v}_3B\tau(D/B)\cbf(q)\mathfrak{q}^{1/3}Z_3(\log
\mathfrak{q})\\
&+\mathfrak{v}_4\cbf(q)\mathfrak{q}^{1/3}Z_4(\log \mathfrak{q})\\
&+\mathfrak{v}_5B\tau(D/B)\cbf(q)\mathfrak{q}^{1/3}Z_5(\log \mathfrak{q}).
\end{split}
\end{equation}

We combine \eqref{star eq 2} and \eqref{star eq 1}, 
and use the inequality
$\sqrt{x+y}\le \sqrt{x}+\sqrt{y}$ with $x,y\ge 0$. This gives
\begin{equation}
\sqrt{(\ell_0+1)\sum_{0\le \ell < \ell_0} (*_{\ell}+**_{\ell})} 
\le \sqrt{\cbf(q)} 
\left(\sqrt{Z_6(\log
\mathfrak{q})}+\sqrt{B\tau(D/B)Z_7(\log\mathfrak{q})}\right)
\mathfrak{q}^{1/6}
\end{equation}
where
\begin{equation}
\begin{split}
&Z_6(X):= \mathfrak{v}_1Z_1(X)+\mathfrak{v}_2Z_2(X)+\mathfrak{v}_4Z_4(X),\\
&Z_7(X):= \mathfrak{v}_3Z_3(X)+\mathfrak{v}_5Z_5(X).
\end{split}
\end{equation}
Finally, we substitute this back into \eqref{main theorem 0} which yields 
\begin{equation}\label{main theorem region 1}
\begin{split}
\left|\sum_{0\le \ell < \ell_0} 
 \sum_{\rho^{\ell}v_0 \le n < \rho^{\ell+1} v_0}
 \frac{\chi(n)}{n^{1/2+it}}\right|
&\le \frac{2\nu_2}{\sqrt{\pi}} \sqrt{\Lambda \cbf(q) Z_6(\log
\mathfrak{q})} \mathfrak{q}^{1/6}\\
&+ \frac{2\nu_2}{\sqrt{\pi}} \sqrt{\Lambda \cbf(q)
 B\tau(D/B)Z_7(\log\mathfrak{q})}\mathfrak{q}^{1/6}.
\end{split}
\end{equation}

\subsection{Sum over $\ell_0 \le \ell < \ell_1$} 
Applying the triangle inequality 
and partial summation gives
\begin{equation}\label{main theorem 6}
\sum_{\ell_0\le \ell<\ell_1}
\left|\sum_{\rho^{\ell} v_0\le n < \rho^{{\ell}+1}v_0} 
\frac{\chi(n)}{n^{1/2+it}}\right| \le 
\sum_{\ell_0\le \ell<\ell_1}
\frac{1}{(\rho^{\ell}v_0)^{1/2}}\sum_{r=0}^{R_{\ell}-1} \max_{0\le \Delta \le L_{\ell}}\left| 
\sum_{n=N_{r,\ell}+1}^{N_{r,\ell}+ \Delta} 
\frac{\chi(n)}{n^{it}}\right|.
\end{equation}
We bound the inner sum in \eqref{main theorem 6} via
Lemma~\ref{corput lemma 1 1}.
Using a similar analysis as in the beginning of Section 5.2, 
one verifies that the required analyticity conditions
on $f(x)=-\frac{t}{2\pi}\log x$
in Lemma~\ref{corput lemma 1 1} hold with $J=1$, $\lambda=1/\sqrt{\rho-1}$,
and $\eta=1/4\pi$.
Therefore, if we let
$\nu_1=\nu_1(1/\sqrt{\rho-1},1/4\pi)$ and 
\begin{equation}
x_{r,\ell}:=\frac{q f'(N_{\ell,r}+1)}{B_1} = -\frac{q t}{2\pi B_1}
\frac{1}{N_{r,\ell}+1},
\end{equation}
then by Lemma~\ref{corput lemma 1 1} 
the inner double sum in \eqref{main theorem 6} is bounded by 
\begin{equation}\label{main theorem 7}
\begin{split}
&\frac{2\nu_1}{\pi}
\frac{C_1R_{\ell}}{(\rho^{\ell}v_0)^{1/2}}
\left(\log\frac{D_1}{2B_1}+\frac{7}{4}+\frac{\pi}{2}\right)\\
&+ \frac{\nu_1}{\pi}
\frac{C_1}{(\rho^{\ell}v_0)^{1/2}}
\sum_{r=0}^{R_{\ell}-1} \min\left(\frac{\pi B_1
L_{\ell}}{q},\frac{1}{\|x_{r,\ell}\|}\right).
\end{split}
\end{equation}
We bound the sum over $r$ in \eqref{main theorem 7}
using Lemma~\ref{well spacing lemma}. To this end, note that 
\begin{equation}\label{xrell bounds}
\begin{split}
|x_{r+1,\ell} -x_{r,\ell}| \ge \frac{q |t|}{2\pi B_1} \frac{L_{\ell}}
{\lfloor \rho^{\ell+1} v_0\rfloor^2},\qquad
|x_{R_{\ell}-1,\ell}-x_{0,\ell}| \le \frac{q |t|}{2\pi B_1}\frac{\rho-1}{\lfloor
\rho^{\ell+1}v_0\rfloor}.
\end{split}
\end{equation}
Furthermore, since the sequence 
$x_{r,\ell}$ is monotonic in $r$, we may set 
$\delta$ in 
Lemma~\ref{well spacing lemma} to be the lower bound for
$|x_{r+1,\ell} -x_{r,\ell}|$ in \eqref{xrell bounds}, 
set $y-x$ as the upper bound for 
$|x_{R_{\ell}-1,\ell}-x_{0,\ell}|$ in \eqref{xrell bounds},  
and set $P=\pi B_1L_{\ell}/q$. (Note that $P\ge 2$.)
Therefore, applying the lemma, and multiplying out the bracket
in the resulting bound 
$2(y-x+1)(2P+\delta^{-1}\log (eP/2))$, gives
\begin{equation}
\begin{split}
\sum_{r=0}^{R_{\ell}-1} \min\left(\frac{\pi B_1
L_{\ell}}{q},\frac{1}{\|x_{r,\ell}\|}\right) &\le
\left(\frac{2(\rho-1)\lfloor \rho^{\ell+1}v_0\rfloor}{L_{\ell}} + \frac{4\pi
B_1}{q |t|} \frac{\lfloor \rho^{\ell+1}
v_0\rfloor^2}{L_{\ell}}\right)\log \left(\frac{e\pi B_1 L_{\ell}}{2q}\right)\\
&+ \frac{2(\rho-1)|t|L_{\ell}}{\lfloor\rho^{\ell+1}
v_0\rfloor} + \frac{4\pi B_1 L_{\ell}}{q}.
\end{split}
\end{equation}
Using similar inequalities as in Section 5.2, we deduce that
\begin{equation}\label{Lell bound 2}
 \frac{(\rho-1)\rho^{\ell} v_0}{|t|^{1/2}}\le L_{\ell} \le 
 \frac{(\rho-1)\rho^{\ell+1} v_0}{|t|^{1/2}} \le \frac{1}{5}\rho(\rho-1)q|t|^{1/2},\qquad
(\ell_0\le \ell<\ell_1).
\end{equation}
Consequently, since $B_1\le D_1$, we have 
\begin{equation}
\log\left(\frac{e\pi B_1L_{\ell}}{2q}\right) \le 
\log\left(\frac{e\pi \rho(\rho-1)D_1
|t|^{1/2}}{10}\right),\qquad (\ell_0\le \ell<\ell_1).
\end{equation}
Using these inequalities, the formulas 
$\sqrt{q}/D_1=\sqf(q)$,
$C_1/(\sqrt{CD} q^{1/6}) = \spf(q)$,
and executing the geometric sum over $\ell$, 
we therefore conclude that 
\begin{equation}\label{main theorem 9}
\begin{split}
&\sum_{\ell_0\le \ell<\ell_1}
\frac{C_1}{(\rho^{\ell}v_0)^{1/2}}
\sum_{r=0}^{R_{\ell}-1} \min\left(\frac{\pi B_1
L_{\ell}}{q},\frac{1}{\|x_{r,\ell}\|}\right) \le \\
&\mathfrak{v}_8\spf(q) \mathfrak{q}^{1/6}Z_8(\log \mathfrak{q})
+ \mathfrak{v}_9 \sqf(q) B_1 Z_8(\log \mathfrak{q})
+\mathfrak{v}_{10} \spf(q) \mathfrak{q}^{1/6} +\mathfrak{v}_{11} \sqf(q) B_1,
\end{split}
\end{equation}
where 
\begin{equation}
Z_8(X) := X+2\log(e\pi \rho(\rho-1)/10)
\end{equation}
and
\begin{equation}
\begin{split}
&\mathfrak{v}_8:= \frac{\rho^{3/2}}{\sqrt{\rho}-1},\qquad
\mathfrak{v}_9:=\frac{2\pi \rho^{5/2}}{\sqrt{5}(\rho-1)(\sqrt{\rho}-1)},\\
&\mathfrak{v}_{10}:=\frac{2\rho^{3/2}(\rho-1)^2}{\sqrt{\rho}-1},\qquad
\mathfrak{v}_{11}:=\frac{4\pi
\rho^{3/2}(\rho-1)}{\sqrt{5}(\sqrt{\rho}-1)}.
\end{split}
\end{equation}
Furthermore, using the bound
\begin{equation}\label{Rell bound 3}
R_{\ell}\le \rho |t|^{1/2},\qquad (\ell_0\le \ell<\ell_1).
\end{equation}
we obtain
\begin{equation}\label{main theorem 10}
\sum_{\ell_0\le \ell<\ell_1}
\frac{C_1R_{\ell}}{(\rho^{\ell}v_0)^{1/2}}
\left(\log\frac{D_1}{2B_1}+\frac{7}{4}+\frac{\pi}{2}\right)\le
\mathfrak{v}_{12} \spf(q) \mathfrak{q}^{1/6}Z_9(\log\mathfrak{q}),
\end{equation}
where
\begin{equation}
\mathfrak{v}_{12}:= \frac{\rho^{3/2}}{2(\sqrt{\rho}-1)},\qquad
Z_9(X) = X-2\log(2 \sqrt{t_0}) + \frac{7}{2}+\pi.
\end{equation}

We substitute \eqref{main theorem 10} and \eqref{main theorem 9} into
\eqref{main theorem 7} and \eqref{main theorem 6}, which gives (after some
simplification)
\begin{equation}\label{main theorem region 2}
\begin{split}
&\sum_{\ell_0\le \ell<\ell_1}
\left|\sum_{\rho^{\ell} v_0\le n < \rho^{{\ell}+1}v_0} 
\frac{\chi(n)}{n^{1/2+it}}\right| \le\\
&\frac{\nu_1}{\pi}\left(\mathfrak{v}_8\spf(q) \mathfrak{q}^{1/6}Z_8(\log \mathfrak{q})
+ \mathfrak{v}_9 \sqf(q) B_1 Z_8(\log \mathfrak{q})
+\mathfrak{v}_{10} \spf(q) \mathfrak{q}^{1/6} +\mathfrak{v}_{11} \sqf(q) B_1\right)\\
&+ \frac{2\nu_1}{\pi} \mathfrak{v}_{12} \spf(q)
\mathfrak{q}^{1/6}Z_9(\log\mathfrak{q}).
\end{split}
\end{equation}

\subsection{Sum over $\ell_1 \le \ell$}
Like before, we 
apply the triangle inequality, partial summation, and 
Lemma~\ref{exp reduce} with $J=0$, to obtain
\begin{equation}\label{main theorem 12}
\sum_{\ell_1\le \ell}
\left|\sum_{\rho^{\ell} v_0\le n < \rho^{{\ell}+1}v_0} 
\frac{\chi(n)}{n^{1/2+it}}\right| \le 
\sum_{\ell_1\le \ell}
\frac{\nu_0}{(\rho^{\ell}v_0)^{1/2}} \sum_{r=0}^{R_{\ell}-1} \max_{0\le \Delta \le
L_{\ell}}\left| 
\sum_{n=N_{r,\ell}+1}^{N_{r,\ell}+ \Delta} \chi(n)\right|.
\end{equation}
Here, $\nu_0 := \nu_0(1/\sqrt{\rho-1},1/2\pi)$. We use the bound
for non-principal characters at the bottom of page 139 of
\cite{hiary-char-sums}. Specifically, if $\chi\pmod{q}$ is non-principal, 
then $|\sum_n \chi(n) |\le 2\sqrt{q} \log q$. Using this in \eqref{main theorem
12}, we deduce that the r.h.s.\ is 
\begin{equation}\label{main theorem region 3}
\le 2\nu_0\sqrt{q}\log q \sum_{\ell_1\le \ell}
\frac{1}{(\rho^{\ell}v_0)^{1/2}} \le 
\frac{2\sqrt{5}\nu_0\sqrt{\rho}}{\sqrt{\rho}-1} \frac{\log q}{\sqrt{|t|}}
\le \nu_0 \mathfrak{v}_{13}Z_{10}(\log \mathfrak{q}),
\end{equation}
where
\begin{equation}
\mathfrak{v}_{13}:=\frac{2\sqrt{5}
\sqrt{\rho}}{(\sqrt{\rho}-1)\sqrt{t_0}},\qquad Z_{10}(X) := X-\log t_0. 
\end{equation}

\subsection{Summary}
We combine \eqref{main theorem region 0}, 
\eqref{main theorem region 1}, 
\eqref{main theorem region 2}, and
\eqref{main theorem region 3}, then 
evaluate the resulting numerical constants
with $\rho = 1.3$. This yields the theorem.

\section{Proof of Corollary~\ref{main theorem simple}}\label{theorem simple proof}
We may assume that $q>1$, otherwise the corollary follows from the bound 
\eqref{zeta bound} for principal characters.
By Lemma~\ref{postnikov lemma},
if $\chi$ is primitive, then 
$B=B_1=1$. Also, since $q$ is a sixth power then $\sqf(q)=\cbf(q)=1$ and 
$\spf(q)\le 1$.
Therefore, the functions $Z(X)$ and $W(X)$ in Theorem~\ref{main theorem}
satisfy
\begin{equation}
\begin{split}
Z(X)&\le -9.416 +15.6004 X+1.4327\sqrt{\Lambda(D)}  X^{3/2} 
+12.1673\sqrt{\Lambda(D) \tau(D)}X^{3/2},\\ 
W(X)&\le -296.84+ 114.07X,
\end{split}
\end{equation}
where we used that for $X\ge \log (2t_0)$ we have
\begin{equation}
\begin{split}
&65.5619 - 17.1704 X - 2.4781 X^2 + 0.6807 X^3 \le 0.6807 X^3,\\
&-1732 - 817.82 X +71.68 X^2 + 47.57X^3 \le 49.1 X^3.
\end{split}
\end{equation}
These inequalities are verified using \verb!Mathematica!.

It is easy to see that 
$\Lambda(p^a) \le \tau(p^a)$ which, by multiplicativity, implies that
$\sqrt{\Lambda(D)\tau(D)} \le \tau(D)$. 
Furthermore, since $q$ is a sixth power and $q>1$, then
$\tau(D) \le 0.572\tau(q)$ (as can be seen by considering the case
$q=2^{6a}$), $\tau(q) \ge 7$, and $q\ge 2^6$.
Substituting these bounds into the expressions for $Z(X)$ and $W(X)$, we 
verify via \verb!Mathematica! that
\begin{equation}
\begin{split}
Z(X) &\le \tau(q)(-1.3451 + 2.2287 X + 7.2695 X^{3/2}) \le 7.95 \tau(q) X^{3/2},\\
W(X) &\le \tau(q)(-42.4056 +  16.2958 X) \le 16.30 \tau(q) X. 
\end{split}
\end{equation}
holds for $X\ge \log(2^6 t_0)$. Therefore, 
\begin{equation}
|L(1/2+it,\chi)| \le 7.95 \tau(q) \mathfrak{q}^{1/6} \log^{3/2}\mathfrak{q} + 
16.30 \tau(q) \log\mathfrak{q}.
\end{equation}
Finally, using the bound $\mathfrak{q} \ge 2^6t_0$, we deduce that
$|L(1/2+it,\chi)| \le 9.05 \tau(q) \mathfrak{q}^{1/6}\log^{3/2} \mathfrak{q}$, 
proving the corollary.

\section{Proofs of bounds \eqref{partial summation bound} and 
\eqref{convexity bound}}\label{bounds proofs}

\begin{proof}[Proof of bound \eqref{partial summation bound}]
Since $\chi$ is nonprincipal, we have
\begin{equation}
L(1/2+it,\chi)  = \sum_{n\le M}
\frac{\chi(n)}{n^{1/2+it}}+\mathcal{R}_M(t,\chi),
\end{equation}
where the remainder 
$\mathcal{R}_M(t,\chi) := \sum_{n>M}^{\infty} \chi(n) n^{-1/2-it}$
is just the tail of the Dirichlet series. (We do not
require that $M> 0$ be an integer.)
To estimate the tail, we use  
partial summation \cite[formula (1)]{rubinstein-computational-methods}, 
\begin{equation}\label{partial sum 1}
\begin{split}
\left|\sum_{M<n\le M_2} \frac{\chi(n)}{n^{1/2+it}}\right| 
\le& 
\frac{1}{\sqrt{M_2}} \left|\sum_{n\le M_2} \chi(n) \right|+
\frac{1}{\sqrt{M}} \left|\sum_{n\le M} \chi(n)\right|\\ 
& + (1/2+|t|)\int_M^{M_2} \left|\sum_{1\le n\le u} \chi(n)\right| u^{-3/2}\, du.
\end{split}
\end{equation}
We bound the character sums on the r.h.s.\ of \eqref{partial sum 1} 
using the P\'olya-Vinogradov inequality in 
\cite[\textsection{23}]{davenport-book}.
This asserts that
if $\chi$ is a primitive character  modulo $q>1$ then
$|\sum_{N_1\le n< N_2} \chi(n) |\le  \sqrt{q}\log q$.
Substituting this in \eqref{partial sum 1}, taking the limit $M_2\to \infty$, 
and executing the integral gives
\begin{equation}
|L(1/2+it,\chi)|\le  2\sqrt{M} + \frac{2\sqrt{q}\log q}{\sqrt{M}}(|t|+1).
\end{equation}
The claimed bound follows on choosing $M=(|t|+1)\sqrt{q}\log q$.
\end{proof}

\begin{remark}
If $\chi$ is merely assumed to be nonprincipal, then the bound \eqref{partial
summation bound} still holds but with an extra factor of $\sqrt{2}$ in front. 
One simply uses the P\'olya-Vinogradov inequality stated 
in \cite[page 139]{hiary-char-sums} in the proof.
\end{remark}

\begin{proof}[Proof of bound \eqref{convexity bound}]
Since
$|L(1/2+it,\chi)| = |L(1/2-it,\overline{\chi})|$ 
and the proof will apply symmetrically
to $L(1/2+it,\chi)$ and $L(1/2+it,\overline{\chi})$, 
we may assume that $t\ge 0$. 
Let $n_1=\lfloor \sqrt{qt/(2\pi)}\rfloor$. Since $qt\ge 2\pi$, 
\cite[Theorem 5.3]{habsieger} implies that 
\begin{equation}\label{approx func eq formula}
|L(1/2+it,\chi)| \le (2 + \delta_t)\left|\sum_{n=1}^{n_1}
\frac{\chi(n)}{n^{1/2+it}}\right|+|\mathcal{R}(t,\chi)|,
\end{equation}
where\footnote{The appearance of $\delta_t$ in \eqref{approx func eq formula} 
is due to a slight imperfection in the form of the approximate functional
equation proved in \cite{habsieger}, and is not significant otherwise.}
$\displaystyle \delta_t := e^{\frac{\pi}{24t}+\frac{1}{12t^2}}-1$ and 
\begin{equation}\label{calR bound}
|\mathcal{R}(t,\chi)|\le \frac{264.72 q^{1/4}\log q}{t^{1/4}} + \frac{11.39
q^{3/4}}{t^{3/4}}e^{-0.78\sqrt{t/q}}.
\end{equation}
To prove this, 
we specialize \cite[Theorem 5.3]{habsieger}
to the critical line, taking $X=Y$ with $2\pi X^2 = qt$,
then appeal to well-known properties of Gauss sums. Put together, this yields
\begin{equation}\label{L formula}
L(1/2+it,\chi) = \sum_{n\le X} \frac{\chi(n)}{n^{1/2+it}} +
F(t,\chi)\sum_{n\le X} \frac{\overline{\chi(n)}}{n^{1/2-it}}+ \mathcal{R}(t,\chi),
\end{equation}
where $G(\chi,-1)$ is a Gauss sum and  
\begin{equation}\label{F def}
F(t,\chi):=\frac{(2\pi i )^{1/2+it}q^{-1/2-it}G(\chi,-1)}{\Gamma(1/2+it)}. 
\end{equation}
(Here, $(2\pi i )^{1/2+it}$ is defined using the principal branch of the
logarithm.) 
We estimate $\mathcal{R}(t,\chi)$ in \eqref{L formula} using the case ``$X\le Y$'' in  \cite[Theorem 5.3]{habsieger}. Since we specialized $X=\sqrt{qt/(2\pi)}$, we obtain 
\begin{equation}
|\mathcal{R}(t,\chi)| 
\le \left(167.2(2\pi)^{1/4}\log q + \frac{2.87 (2\pi)^{3/4}
\sqrt{q}}{\sqrt{t}}e^{-\sqrt{\pi^3/50} \sqrt{t/q}}\right)\frac{q^{1/4}}{t^{1/4}}.
\end{equation}
The claimed estimate \eqref{calR bound} 
for $\mathcal{R}(t,\chi)$ follows on noting that
$167.2(2\pi)^{1/4} < 264.72$, $2.87(2\pi)^{3/4} < 11.39$, and
$\pi^{3/2}/\sqrt{50} > 0.78$.

To bound the factor $1/\Gamma(1/2+it)$ appearing in the definition of $F(t)$, we mimic the proof of  \cite[Lemma 2.1]{habsieger} with minor adjustments. This gives 
\begin{equation}\label{gamma bound}
\frac{1}{|\Gamma(1/2+it)|} \le \frac{e^{\frac{\pi t}{2}+\frac{\pi}{24t}+\frac{1}{12t^2}}}{\sqrt{2\pi}},\qquad (t>0).
\end{equation}
Combining \eqref{gamma bound} with the facts $|G(\chi,-1)|=\sqrt{q}$ and 
$|(2\pi i )^{1/2+it}| = \sqrt{2\pi}e^{-\pi t/2}$ gives
$|F(t,\chi)|\le e^{\frac{\pi}{24t}+\frac{1}{12t^2}}=1+\delta_t$.
Since the second sum in the approximate functional equation 
\eqref{L formula} is just the complex
conjugate of the first sum there, this proves \eqref{approx func eq
formula}.

Last, we trivially estimate the sum in 
 \eqref{approx func eq formula}, then use the assumption $t\ge \sqrt{q}\ge
 \sqrt{2}$ and monotonicity 
 to bound $\mathcal{R}(t,\chi)$ and $\delta_t$. This gives (on noting
 that $\log q \le \log \mathfrak{q}^{2/3}$)
\begin{equation}
|L(1/2+it,\chi)|\le 
\left(\frac{2(1+\delta_{\sqrt{2}})}{(2\pi)^{1/4}}+
11.39e^{-\frac{0.78}{2^{1/4}}}\right) \mathfrak{q}^{1/4} 
+ 264.72 \mathfrak{q}^{1/12}\log \mathfrak{q}^{2/3}. 
\end{equation}
Denote the r.h.s.\ above by $(*)$. Using \verb!Mathematica! we verify that
the equation $(*)=124.46 \mathfrak{q}^{1/4}$ has no real solution if $\mathfrak{q}
\ge 10^9$. Furthermore, $(*)$ is smaller than $124.46 \mathfrak{q}^{1/4}$ when
$\mathfrak{q}=10^9$. Hence, $(*) \le 124.46 \mathfrak{q}^{1/4}$ 
for all $\mathfrak{q} \ge 10^9$, as claimed.
\end{proof}

\bibliographystyle{amsplain}
\bibliography{hybridDirichlet}
\end{document}